\documentclass[11pt, reqno]{amsart}
\usepackage{amsmath,amsthm,amssymb,mathrsfs}
\usepackage[abbrev,non-sorted-cites]{amsrefs}
\usepackage{color}
\usepackage{verbatim}
\usepackage{datetime}
\usepackage{hyperref}

\theoremstyle{plain}
  \newtheorem{thm}{Theorem}[section]
  \newtheorem{lem}[thm]{Lemma}
  \newtheorem{prop}[thm]{Proposition}
   
\theoremstyle{definition}
  
	\newtheorem{assumption}[thm]{Assumption}
  \newtheorem{rmk}[thm]{Remark}
  \newtheorem*{ack*}{Acknowledgement}
  \newtheorem*{ques*}{Question}
\theoremstyle{plain}

\numberwithin{equation}{section}

\newcommand\ip[2]{\langle{#1},{#2}\rangle}
\newcommand\RR[4]{R(\bar{e}_{#1},\bar{e}_{#2},\bar{e}_{#3},\bar{e}_{#4})}

\newcommand\pl{\partial}

\newcommand\oh{\frac{1}{2}}
\newcommand\dd{{\mathrm d}}
\newcommand\ii{{\mathrm i}}

\newcommand\w{\wedge}
\newcommand\sm{\sigma}
\newcommand\dt{\delta}
\newcommand\ep{\epsilon}
\newcommand\om{\omega}
\newcommand\ta{\theta}
\newcommand\gm{\gamma}
\newcommand\kp{\kappa}
\newcommand\af{\alpha}
\newcommand\bt{\beta}

\newcommand\Om{\Omega}
\newcommand\Sm{\Sigma}

\newcommand\CC{\mathcal{C}}
\newcommand\CO{\mathcal{O}}
\newcommand\CH{\mathcal{H}}
\newcommand\CV{\mathcal{V}}
\newcommand\CA{\mathcal{A}}
\newcommand\CI{I}

\newcommand\FR{\mathfrak{R}}

\newcommand\RO{\mathrm{O}}
\newcommand\RU{\mathrm{U}}

\newcommand\gt{\mathrm{G}_2}
\newcommand\sps{\mathrm{Spin}(7)}

\newcommand\BS{\mathbb{S}}
\newcommand\BC{\mathbb{C}}
\newcommand\BCP{\mathbb{CP}}
\newcommand\BR{\mathbb{R}}

\newcommand\fs{\mathfrak{s}}
\newcommand\fu{\mathfrak{u}}

\newcommand\ul{\underline}

\newcommand\br{\bar}
\newcommand\ol{\overline}

\newcommand\ot{\otimes}

\newcommand\da{\dot{a}}
\newcommand\db{\dot{b}}
\newcommand\dc{\dot{c}}

\DeclareMathOperator{\tr}{tr}

\DeclareMathOperator{\Hess}{Hess}
\DeclareMathOperator{\End}{End}
\DeclareMathOperator{\rank}{rank}
\DeclareMathOperator{\re}{Re}
\DeclareMathOperator{\im}{Im}

\begin{document}

\title{Mean curvature flows in manifolds of special holonomy}

\author{Chung-Jun Tsai}
\address{Department of Mathematics\\
National Taiwan University\\ Taipei 10617\\ Taiwan}
\email{cjtsai@ntu.edu.tw}

\author{Mu-Tao Wang}
\address{Department of Mathematics\\
Columbia University\\ New York\\ NY 10027\\ USA}
\email{mtwang@math.columbia.edu}

\date{\usdate{\today}}

\thanks{Supported in part by Taiwan MOST grants 102-2115-M-002-014-MY2 and 104-2115-M-002-007 (C.-J.\ Tsai) and NSF grants DMS-1105483 and DMS-1405152 (M.-T.\ Wang).  The second-named author is supported in part by a grant from the Simons Foundation (\#305519 to Mu-Tao Wang). Part of this work was carried out when Mu-Tao Wang was visiting the National Center of Theoretical Sciences at National Taiwan University in Taipei, Taiwan. 
}

\begin{abstract}
We study the uniqueness of minimal submanifolds and the stability of the mean curvature flow in several well-known model spaces of manifolds of special holonomy. These include the Stenzel metric on the cotangent bundle of spheres, the Calabi metric on the cotangent bundle of complex projective spaces, and the Bryant-Salamon metrics on vector bundles over certain Einstein manifolds. In particular, we show that the zero sections, as calibrated submanifolds with respect to their respective ambient metrics,  are unique among compact minimal submanifolds and are dynamically stable under the mean curvature flow.   The proof relies on intricate interconnections of the Ricci flatness of the ambient space and the extrinsic geometry of the calibrated submanifolds.

\end{abstract}

\maketitle


\section{Introduction}\label{sc_intro}

Calibrated submanifolds \cite{ref_HL} in manifolds of special holonomy are not just minimal submanifolds, they actually minimize the volume  functional in their homology classes. Of particular interests are special Lagrangians in Calabi--Yau, associatives and coassociatives in $\gt$, and Cayley submanifolds in $\sps$. These geometric objects attracted a lot of attentions in recent years. On the one hand, they are natural generalizations of algebraic subvarieties in algebraic manifolds and thus are of immense geometric interest.  On the other hand, they appear in various proposals of string theory such as Mirror Symmetry and the M-theory.  The most successful construction of metrics of special holonomy is the Calabi--Yau case, where the celebrated theorem of Yau \cite{ref_Yau} shows the homological condition guarantees the existence of the metric. All other constructions are based on deformation theory,  symmetry reductions or gluing constructions \cite{ref_Joyce} (see also a recent flow approach for $\gt$ construction in \cites{ref_BX, ref_LW}). The scenario of the construction of calibrated submanifolds is similar \cites{ref_IKM, ref_KM}.  For special Lagrangians in Calabi--Yau's, we refer to the work of Schoen--Wolfson \cite{ref_ScW} and Joyce \cites{ref_Joyce2, ref_Joyce3, ref_Joyce4}.

Among all explicitly constructed manifolds of special honolomy, the most well-known ones seem to be the Stenzel metric \cite{ref_St} on the cotangent bundles of spheres (or the Eguchi--Hanson metric in dimension $4$) and the Calabi metric \cite{ref_Calabi} on the cotangent bundles of complex projective spaces.  Similar constructions of Bryant--Salamon \cite{ref_BS} produce $\gt$ and $\sps$ metrics. All of them are based on bundle constructions and the zero sections are calibrated submanifolds. In the article, we study the uniqueness and the dynamical stability of the zero sections of these manifolds. To be more specific, we consider the following manifolds of special honolomy in this article.

\begin{assumption}\label{model_space} Throughout this article, $M$ is a Riemannian manifold of special holonomy that belongs to one of the followings:
\begin{enumerate}

\item the total space of the cotangent bundle of $S^{n}$ with the Stenzel metric\footnote{When $n=1$, the metric is not only Ricci flat, but flat.  We have to exclude this flat case.} with $n>1$;

\item the total space of the cotangent bundle of $\BCP^n$ with the Calabi metric;

\item the  total space of $\BS(S^3)$, $\Lambda^2_-(S^4)$, $\Lambda^2_-(\BCP^2)$, or $\BS_-(S^4)$ with the Ricci flat metric constructed by Bryant--Salamon.
\end{enumerate}
\end{assumption}

In section \ref{sc_Stenzel}, \ref{sc_Calabi} and \ref{sc_BS}, we review the geometry of these metrics in details. 
In all these examples, $M$ is the total space of a vector bundle over a base manifold $B$. We identify $B$ with the zero section of the bundle, which is also considered to be an embedded submanifold of $M$.  In each case, there exists a smooth differential form $\Om$ with the following properties:
\begin{itemize}
\item $\Om$ has comass one, i.e.\ at any $p\in M$, $\Om(L)\leq1$ for any (oriented) subspace $L\subset T_p M$ with dimension $=\dim B$. Indeed, $\Om$ is locally the wedge product of orthonormal $1$-forms. 
\item The form $\Om$ characterizes $B$ by the condition that $\Omega(T_pB)=1$ for all $p\in B$ where $T_pB$ is the (oriented) tangent space of $B$ as a submanifold of $M$ at a point $p\in B$.
\end{itemize}
The precise definition of $\Om$ in each case can be found in \eqref{Stenzel_Om}, \eqref{hk_Om} and \eqref{bd_Om}, respectively.
Let $d(\,\cdot\,)$ denote the distance function to the zero section $B$ with respect to the Riemmanian metric on $M$. For a compact embedded submanifold $\Sigma$ of $M$ with $\dim \Sm = \dim B$, $\Sigma$ is considered to be $\CC^0$ close to the zero section if $d(p)$ is close to $0$ for all $p\in \Sigma$, and $\Sigma$ is considered to be $\CC^1$ close to the zero section if $\Omega(T_p\Sigma)$ is close to $1$ for all $p\in \Sigma$.

Our first result regards the uniqueness property of the zero sections.

\begin{thm}\label{unique}
In each case considered in Assumption \ref{model_space}, the zero section is the unique compact minimal immersed submanifold of the given dimension. 
\end{thm}

The mean curvature flow is the parabolic PDE system that deforms a submanifold by its mean curvature vector field, and is formally the negative gradient flow of the volume functional. A calibrated submanifold represents a local minimum of the volume functional.
It is therefore natural to investigate the stability of a calibrated submanifold along the mean curvature flow, which is a nonlinear degenerate PDE system. The nonlinear stability of PDE systems such as the Einstein equation or the Ricci flow is under intense study. Our next result concerns the nonlinear $\CC^1$ stability of the zero sections. 

\begin{thm}\label{stable}
In each case considered in Assumption \ref{model_space}, there exists an $\epsilon>0$ which depends only on the geometry of $M$ such that if $\Sigma$ is a compact embedded submanifold of $M$ and
\begin{align} \label{condition0}
\sup_{p\in \Sm} \left( d({p})+ 1- \Om(T_p\Sm) \right)<\ep ~,
\end{align}
then the mean curvature flow of $\Sigma$ exists for all time and converges to the zero section smoothly. 
\end{thm}

Note that $(1- \Om(T_p\Sm))$ is always non-negative since $\Omega$ has comass one. 
\begin{rmk}
The Ricci flat metric on each manifold $M$ considered in Assumption \ref{model_space} is constructed under some symmetry ansatz.  From the constructions, one sees that the metric turns out to be uniquely determined by the volume of the zero section $B$.  That is to say, $\epsilon$ depends on the volume of the zero section (and also $n$ in case (i) and (ii)).
\end{rmk}

Theorem \ref{unique} is proved in section \ref{sc_rigidity}.  The main point is to prove the convexity of the distance square to the zero section.  It is also the key for the $\CC^0$ convergence in Theorem \ref{stable}, and plays an important role for the $\CC^1$ convergence.

Such a long-time existence and convergence theorem under effective $\CC^1$ bound for higher codimensional mean curvature flows has been established for manifolds of reduced holonomy, namely manifolds that are locally Riemannian products, see for examples \cites{ref_SW, ref_TW, ref_W1, ref_W2, ref_W4}. Theorem \ref{stable}, to the best our knowledge,  appears to be the first one for manifolds of special holonomy. 

Sections \ref{sc_estimate} and \ref{sc_stability} are devoted to the proof of Theorem \ref{stable}.  In section \ref{sc_estimate}, we establish estimates on the covariant derivatives of $\Om$, which is needed for proving the $\CC^1$ convergence.  In section \ref{sc_stability}, we put everything together to prove the stability of the zero section under the mean curvature flow. 

\begin{rmk}
A stability theorem of the mean curvature flow of compact smooth submanifolds can be derived from Simon's general stability theorem for gradient flows \cite{ref_Simon}*{Theorem 2} under the assumption that (1) the initial data is close enough to a stable minimal submanifold in the  $W^{l+2, 2}$ Sobolev norm  for large enough $l$ (which implies at least $\CC^2$ smallness) and (2) the ambient metric is analytic.  Theorem \ref{stable} is a $\CC^1$ stability theorem in which the regularity requirement is lower and the dependence of the smallness constant is explicit.\footnote{The authors learned from Felix Schulze that it is possible to derive a $\CC^1$ stability theorem from White's regularity theorem, the uniqueness of mean curvature flows, and a limiting argument. However, it seems that the $\CC^1$ bound exists by an argument of contradiction and cannot be made explicit.} The proof does not rely on the analyticity of the ambient metric either. The theorem can be turned into a Lipschitz stability theorem (by approximating the initial Lipschitz submanifold by a family of $\CC^1$ submanifolds, see for example \cite{ref_W5}), which seems to be the optimal result for the mean curvature flow.
\label{rmk_intro} \end{rmk}

\subsection{Notations and conventions}

\subsubsection{Riemann curvature tensor}\label{sc_convection1}

In this paper, for a Riemannian manifold with metric $\ip{\,}{\,}$ and Levi-Civita connection $\nabla$, our convention for the Riemann curvature tensor is
\begin{align}
R(X,Y,Z,W) = \ip{\nabla_Z\nabla_W Y - \nabla_W\nabla_Z Y - \nabla_{[Z,W]} Y}{X} ~.
\label{R_curv} \end{align}

Let $\{e_i\}$ be a local orthonormal frame.  Denote the dual coframe by $\{\om^i\}$, and their connection $1$-forms by $\om_i^j$.  The convention here is
\begin{align}
\nabla e_i = \om_i^j\ot e_j \quad\text{ and }\quad \nabla\om^j = -\om_i^j\ot\om^i ~.
\label{str1} \end{align}
Throughout this paper, we adopt the Einstein summation convention that repeated indexes are summed.  Since the frame is orthonormal, $\om_i^j = -\om_j^i$.  It follows from \eqref{str1} that
\begin{align}
\dd\om^j = -\om_i^j\w\om^i ~.
\label{str2} \end{align}

Its curvature form is
\begin{align}
\FR_i^j = \dd\om_i^j - \om_i^k\w\om_k^j ~.
\label{R_curv1} \end{align}
It is equivalent to the Riemann curvature tensor by the following relation:
\begin{align}
\FR_i^j(X,Y) = R(e_j,e_i,X,Y)
\label{R_curv2} \end{align}
for any two tangent vectors $X$ and $Y$.

\subsubsection{Bundle projection}\label{sc_convection2}

Let $\pi:M\to B$ be a vector bundle projection and $\Psi$ be a (locally-defined) smooth differential form on $B$.  The following abuse of notation is performed throughout this paper:  the pull-back of $\Psi$, $\pi^*\Psi$ on $M$ , is still denoted by $\Psi$.

\begin{ack*}
The first author would like to thank Prof. S.-T. Yau for bringing the rigidity question into his attention, and for his generosity in sharing his ideas.  The authors would like to thank Mao-Pei Tsui and Felix Schulze for helpful discussions and interests in this work. 
\end{ack*}

\section{Geometry of the Stenzel metric} \label{sc_Stenzel}

\subsection{The Stenzel metric on $T^*S^n$}

Consider the $n$-dimensional sphere $S^n$ with the standard metric for $n>1$.  Let $\{\ul{\om}^\mu\}_{\mu=1}^n$ be a local orthonormal coframe, and $\ul{\om}_\nu^\mu$ be their connection $1$-forms which satisfy \eqref{str2}. 
As a space form with curvature equal to $1$, its curvature form is
\begin{align}
\ul{\FR}_\nu^\mu &= \dd\ul{\om}_\nu^\mu - \ul{\om}_\nu^\gm\w\ul{\om}_\gm^\mu = \ul{\om}^\mu\w\ul{\om}^\nu ~.
\end{align}

Let $\{y_\mu\}$ be the coordinate for the fibers of $T^*S^n$ induced by $\{\ul{\om}^\mu\}$.  The standard metric on $S^n$ induces the following metric on $T^*S^n$.
\begin{align} \label{metric0}
\sum_{\mu=1}^n \left( (\ul{\om}^\mu)^2 + (\dd y_\mu - y_\nu\,\ul{\om}_\mu^\nu)^2 \right) ~,
\end{align}
where the abuse of notation \ref{sc_convection2} is adopted.

\subsubsection{Spherical coordinate for the fibers}
There are two naturally defined $1$-forms outside the zero section:
\begin{align*}
\frac{1}{r} y_\mu\,\ul{\om}^\mu   \qquad\text{and}\qquad
\frac{1}{r}y_\mu\,\dd y_\mu = \frac{1}{r}y_\mu\,(\dd y_\mu - y_\nu\,\ul{\om}^\nu_\mu)
\end{align*}
where $r = (\sum_\mu (y_\mu)^2)^\oh$.  The first one is the tautological $1$-form rescaled by $1/r$; the second one is the exterior derivative of $r$.  With respect to \eqref{metric0}, they are unit-normed and orthogonal to each other.

It is more convenient to consider the metric \eqref{metric0} by another coframe, which is an extension of the above two $1$-forms.  To start, extend the vector $\frac{1}{r}(y_1,\cdots,y_n)$ to an orthonormal frame for $\BR^n$.  This can be done on any simply-connected open subset of $\BR^n\backslash\{0\}$.  To be more precise, choose a smooth map $T_\nu^\mu$ from a simply-connected open subset of $\BR^n\backslash\{0\}$ to $\RO(n)$ such that
\begin{align*}
T_\nu^1 (y) = \frac{1}{r} y_\nu \qquad\text{for}\quad\nu=1,2,\ldots,n \quad\text{and}\quad y\in \text{Domain}(T_\nu^\mu)\subset\BR^n\backslash\{0\}.
\end{align*}

For example, the standard spherical coordinate system on $\BR^n\backslash\{0\}$ will do. When $n=3$ with $y_1=r\sin\theta\sin\phi, y_2=r\sin\theta \cos\phi, y_3=r\cos\theta$, one can take 
$T_\mu^2=\frac{1}{r} \frac{\pl y_\mu}{\pl \theta} $ and $T_\mu^3=\frac{1}{r\sin\theta} \frac{\pl y_\mu }{\pl \phi}$, $\mu=1, 2, 3$.

The metric \eqref{metric0} has the following orthonormal coframe
\begin{align} \begin{split}
\sm^\mu &= T_\nu^\mu\,\ul{\om}^\nu ~, \\
\sm^{n+\mu} &= (T^{-1})_\mu^\nu\,(\dd y_\nu - y_\gm\,\ul{\om}_\nu^\gm) = T_\nu^\mu\,(\dd y_\nu - y_\gm\,\ul{\om}_\nu^\gm)
\end{split} \label{frame0} \end{align}
for $\mu\in\{1,\ldots,n\}$.

\subsubsection{The Stenzel metric}
With this coframe \eqref{frame0}, the Stenzel metric takes the form
\begin{align}
\left( c(r)\,\sm^1 \right)^2 + \sum_{j=2}^n \left( a(r)\,\sm^j \right)^2
+ \left( c(r)\,\sm^{n+1}\right)^2 + \sum_{j=2}^n \left( \frac{b(r)}{r}\,\sm^{n+j} \right)^2 ~.
\label{metric1} \end{align}
The coefficient functions are defined by
\begin{align} \begin{split}
a^2 &= \frac{1}{4}h'(r)\coth r ~, \\
b^2 &= \frac{1}{4}h'(r)\tanh r ~, \\
c^2 &= \frac{1}{4}h''(r)
\end{split} \label{coeff0} \end{align}
where $h'(r)$ is the solution of the ODE
\begin{align}
\frac{\dd}{\dd r} \left(h'(r)\right)^n &= 2^{n+1} n  \left(\sinh(2r)\right)^{n-1}
\qquad\text{with}\quad h'(0) = 0 ~.
\label{Ricci0} \end{align}
Here, prime $(\,\,\,)'$ denotes the derivative with respect to $r$.  The function $h$ is the K\"ahler potential in the paper of Stenzel \cite{ref_St}*{section 7}.  We remark that the metric here differs from that in \cite{ref_CGLP}*{section 2} by a factor of $2^{n-1}/n$ in \eqref{Ricci0}.  The normalization here is chosen such that the restriction of metric \eqref{metric1} to the zero section is the metric of the round sphere of radius $1$ and dimension $n$, and hence the volume of the zero section is $2\pi^{\frac{n}{2}}/\Gamma(n/2)$.

\subsubsection{Connection and Ricci flat equation}
Take the orthonormal coframe:
\begin{align}
\om^1 &= c(r)\,\sm^1~,    &\om^j &= a(r)\,\sm^j ~,
&\om^{n+1} &= c(r)\,\sm^{n+1} = c(r)\,\dd r ~,    &\om^{n+j} &= \frac{b(r)}{r}\,\sm^{n+j}
\label{frame1} \end{align}
for $j\in\{2,\ldots,n\}$.  The indices $i,j,k,\ldots$ will be assumed to belong to $\{2,\ldots,n\}$.

It is useful to introduce the new radial function $\rho$ by $\rho = \int_0^r c(u)\dd u$.  Since $\dd\rho = c(r)\,\sm^{n+1}$ and $\rho=0$ at the zero section, in view of \eqref{metric1}, $\rho$ is the geodesic distance to the zero section with respect to the Stenzel metric.  Denoting by dot $\dot{(\,\,\,)}$ the derivative with respect to $\rho$, we have 
\begin{align*}
\dot{f} &= \frac{1}{c(r)} f' ~.
\end{align*}
The connection $1$-forms $\om_\nu^\mu$ of \eqref{frame1} can be found by a direct computation:
\begin{align} \begin{split}
\om^1_{n+1} &= \frac{\dc}{c}\,\om^1 ~,\\
\om^j_{n+i} &= C\delta_i^j\,\om^1 ~,
\end{split} &\begin{split}
\om^{j}_{n+1} &= \frac{\da}{a}\,\om^j ~,\\
\om^{n+j}_1 &= A\,\om^j ~,
\end{split} &\begin{split}
\om^{n+j}_{n+1} &= \frac{\db}{b}\,\om^{n+j} ~,\\
\om^1_j &= B\,\om^{n+j}
\end{split} \label{conn0} \end{align}
where
\begin{align}
A &= \frac{a^2-b^2-c^2}{2abc} ~,
&B &= \frac{b^2-a^2-c^2}{2abc} ~,
&C &= \frac{c^2-a^2-b^2}{2abc} ~.
\label{coeff1} \end{align}
The rest of the components $\om_i^j$ and $\om_{n+i}^{n+j}$ satisfy
\begin{align}
\om_i^j &= \om_{n+i}^{n+j} = T^j_\mu\,\ul{\om}^\mu_\nu\,T^i_\nu - (\dd T^j_\mu)\,T^i_\mu ~.
\label{conn1} \end{align}

In \cite{ref_CGLP}*{section 2}, Cveti{\v{c}} et al.\ derived the Stenzel metric in a different way.  As a result of their derivations, $a, b, c$ satisfy the following differential system:
\begin{align}
\frac{\da}{a} + A &= 0 ~,
&\frac{\db}{b} + B &= 0 ~,
&\frac{\dc}{c} + (n-1)C &= 0 ~.
\label{Ricci1} \end{align}
One can also check these directly by \eqref{coeff0}, \eqref{Ricci0} and \eqref{coeff1}.  In fact,  Cveti{\v{c}} et al.\ solved the system \eqref{Ricci1} and reconstructed the Stenzel metric in the above form \eqref{metric1}.  The expressions of the K\"ahler form and the holomorphic volume form are quite simple in terms of the coframe \eqref{frame1}.  The K\"ahler form is $\sum_{\mu=1}^n \om^\mu\w\om^{n+\mu}$, and the holomorphic volume form is $(\om^1+\ii\,\om^{n+1})\w(\om^2+\ii\,\om^{n+2})\w\cdots(\om^{n}+\ii\,\om^{2n})$.  With the above relations, one can check that these two differential forms are parallel. { With this understanding, the complex structure $\CI$ in terms of the dual frame $\{\bar{e}_1,\ldots,\bar{e}_{2n}\}$ of \eqref{frame1} sends $\bar{e}_\mu$ to $\bar{e}_{n+\mu}$ and sends $\bar{e}_{n+\mu}$ to $-\bar{e}_\mu$ for any $\mu\in\{1,\ldots,n\}$.}

\subsection{Coefficient functions and curvature}
The geometry of the Stenzel metric is encoded in the functions $A$, $B$ and $C$.  In this subsection, we summarize the properties that will be used later.  Note that \eqref{Ricci0} implies that $h'>0$ and $h''>0$ when $r>0$.  By \eqref{coeff0} and \eqref{Ricci1},
\begin{align} \begin{split}
A &= - \frac{1}{\sqrt{h''}} \left( \frac{h''}{h'} - \frac{2}{\sinh(2r)} \right) ~,\\
B &= - \frac{1}{\sqrt{h''}} \left( \frac{h''}{h'} + \frac{2}{\sinh(2r)} \right) ~,\\
C&= -\frac{1}{n-1} \frac{1}{\sqrt{h''}} \frac{h'''}{h''} ~.
\end{split} \label{coeff2} \end{align}

Instead of $r$, we state the estimates in terms of $\rho = \int_0^r c(u)\dd u$.
\begin{lem} \label{lem_est1}
The functions $A$, $B$ and $C$ are negative when $\rho>0$.  Moreover, there exists a constant $K>1$ which depends only on $n$ such that $|A|/\rho$, $|C|/\rho$ and $|B|\rho$ are all bounded between $1/K$ and $K$ for any point $p$ with $0<\rho (p)<1$.
\end{lem}

\begin{proof}
It follows from \eqref{Ricci0} that
\begin{align} \begin{split}
\frac{h''}{h'} &= \frac{\big( \sinh(2r) \big)^{n-1}}{n \int_0^r \big(\sinh(2u)\big)^{n-1} \dd u} ~, \\
\frac{1}{n-1}\frac{h'''}{h''} &= \frac{2n\cosh(2r) \left(\int_0^r\big(\sinh(2u)\big)^{n-1}\dd u\right) - \big( \sinh(2r) \big)^n}{n \sinh(2r) \left( \int_0^r\big( \sinh(2u) \big)^{n-1}\dd u \right)} ~.
\end{split} \label{coeff3} \end{align}

To prove that $C<0$, we estimate
\begin{align*}
2n\cosh(2r) \left( \int_0^r \big(\sinh(2u)\big)^{n-1}\dd u \right) &> 2n \int_0^r \cosh(2u) \big(\sinh(2u)\big)^{n-1}\dd u = \big(\sinh(2r)\big)^{n}
\end{align*}
for any $r>0$.  It follows from \eqref{coeff3} that $C$ is negative.  For $A<0$,
\begin{align*}
\frac{h''}{h'} - \frac{2}{\sinh(2r)} &= \frac{\big(\sinh(2r)\big)^n - 2n \int_0^r \big(\sinh(2u)\big)^{n-1} \dd u}{n\sinh(2r) \int_0^r\big( \sinh(2u) \big)^{n-1}\dd u} \\
&= \frac{2n \int_0^r \big(\sinh(2u)\big)^{n-1} \big( \cosh(2u)-1 \big) \dd u}{n\sinh(2r) \int_0^r\big( \sinh(2u) \big)^{n-1}\dd u} > 0
\end{align*}
for any $r>0$.  It follows that $A<0$.  Since $B\leq A$, $B$ is also negative.

The second assertion is a direct consequence of the power series expansion at $r=0$.  It follows from \eqref{Ricci0} that the Taylor series expansion of $h'(r)$ near $r=0$ is
\begin{align*}
h' = 4r \left( 1 + \frac{2(n-1)}{3(n+2)}r^2 + \CO(r^3) \right) ~.
\end{align*}
It then follows that $\rho = r+\CO(r^2)$, $A = -\frac{n}{n+2}r + \CO(r^2)$, $C = -\frac{2}{n+2}r + \CO(r^2)$ and $B = -r^{-1} + \CO(1)$.  This finishes the proof of the lemma.
\end{proof}

We remark that \eqref{Ricci1} and the negativity of $A$, $B$ and $C$ imply that $a$, $b$ and $c$ are increasing functions.  Other quantities we will encounter are the derivatives of $A$, $B$ and $C$ with respect to $\rho$.  It turns out that they are degree two polynomials in $A$, $B$ and $C$.

\begin{lem} \label{lem_est2}
The functions $\dot{A}$, $\dot{B}$ and $\dot{C}$ obey:
\begin{align*}
\dot{A} &= -nBC + A(A+B+C) ~, \\
\dot{B} &= -nAC + B(A+B+C) ~, \\
\dot{C} &= -2AB + (n-1)C(A+B+C) ~.
\end{align*}
\end{lem}

\begin{proof}
The equation \eqref{coeff1} can be rewritten as
\begin{align}
B+C &= -\frac{a}{bc} ~,  &C+A &= -\frac{b}{ac} ~,  &A+B &= -\frac{c}{ab} ~.
\label{coeff4} \end{align}
By using \eqref{Ricci1},
\begin{align*}
\dot{B} + \dot{C} &= (B+C) \left( -A + B + (n-1) C \right) ~,\\
\dot{C} + \dot{A} &= (C+A) \left( -B + A +(n-1) C \right) ~, \\
\dot{A} + \dot{B} &= (A+B) \left( -(n-1) C + A + B \right) ~.
\end{align*}
The lemma follows from these formulae.
\end{proof}

\subsubsection{The Riemann curvature tensor} \label{sc_Stenzel_cruv}

In \cite{ref_CGLP}*{section 2}, Cveti{\v{c}} et al.\ also computed the components of the Riemann curvature tensor of the metric \eqref{metric1}.

Note that \eqref{coeff4} implies that
\begin{align*}
(A+B)(A+C) &= \frac{1}{a^2} ~,  &(B+A)(B+C) &= \frac{1}{b^2} ~,  &(C+A)(C+B) &= \frac{1}{c^2} ~.
\end{align*}
By applying \eqref{conn0}, \eqref{Ricci1}, \eqref{coeff4} and Lemma \ref{lem_est2} to \cite{ref_CGLP}*{(2.14)}, we find that the components of the Riemann curvature tensor are all degree two polynomials in $A$, $B$ and $C$:
\begin{align*}
\RR{1}{j}{1}{j} 
&= AB + BC - nCA ~, \\
\RR{1}{n+j}{1}{n+j} 
&= AB + CA - nBC ~, \\
\frac{1}{n-1}\RR{1}{n+1}{1}{n+1} &= \RR{1}{n+1}{n+j}{j} \\
&= (n-1)(CA + BC) - 2AB ~, \\
\RR{i}{n+k}{j}{n+l} &= -AB(\dt_{ij}\dt_{kl}+\dt_{il}\dt_{jk}) + (BC+AC)\dt_{ik}\dt_{jl}~, \\
\RR{i}{k}{j}{l} 
&= (AB + BC + CA)(\dt_{ij}\dt_{kl} - \dt_{il}\dt_{jk})
\end{align*}
and other inequivalent components vanish.  Here, $\bar{e}_1,\ldots,\bar{e}_{2n}$ is the dual frame of \eqref{frame1}.
By equivalence we mean that the curvature is a quadrilinear map $R$ satisfying the condition
\begin{align}
R(X,Y,Z,W) = -R(X,Y,W,Z) = R(Z,W,X,Y) = R(X,Y,\CI Z,\CI W) ~.
\label{sym1} \end{align}

The following property of the Riemann curvature tensor will help simplify the calculation in the Stenzel metric case:
$ R(\bar{e}_\mu,\bar{e}_{n+\nu},\bar{e}_{n+\delta} ,\bar{e}_{n+\epsilon}) = 0 = R(\bar{e}_{n+\mu}  ,\bar{e}_{\nu},\bar{e}_{\delta} ,\bar{e}_{\epsilon})$ for any $\mu, \nu, \delta, \epsilon \in\{1,\ldots,n\}$.  It will be used for the estimate \eqref{est_mixed} in the proof of Theorem \ref{stable}.

\section{Geometry of the Calabi metric}\label{sc_Calabi}

The Eguchi--Hanson metric has another higher dimensional generalization.  There exists a hyper-K\"ahler metric on the cotangent bundle of the complex projective space, $T^*\BCP^n$.  The metric is thus Ricci flat.  It was constructed by Calabi in \cite{ref_Calabi}*{section 5} by solving the K\"ahler potential under an ansatz.  Since we are going to study the Riemannian geometric properties of the metric, it is more convenient to describe the metric in terms of a moving frame.

\subsection{The Calabi metric on $T^*\BCP^n$}

Consider the $n$-dimensional complex projective space $\BCP^n$ with the Fubini--Study metric.  Let $\{{\ta}^\mu\}$ be a local unitary coframe of type $(1,0)$.  That is to say, the Fubini--Study metric is
\begin{align*}
\sum_{\mu=1}^n | {\ta}^\mu |^2 ~.
\end{align*}
Denote by ${\ta}_\nu^\mu$ the corresponding connection $1$-forms.  They are determined uniquely by the relations:
\begin{align} \label{hk_conn1}
\dd{\ta}^\mu = -{\ta}_\nu^\mu\w{\ta}^\nu   \qquad\text{and}\qquad
{\ta}_\mu^\nu + \ol{{\ta}_\nu^\mu} =0 ~.
\end{align}
The curvature of the Fubini--Study metric is
\begin{align*}
\Theta_\nu^\mu &= \dd\ta_\nu^\mu - \ta_\nu^\gm\w\ta_\gm^\mu = \ta^\mu\w\ol{\ta^\nu} + \delta_\nu^\mu\,\ta^\gm\w\ol{\ta^\gm} ~.
\end{align*}
The curvature formula implies that its sectional curvature lies between $1$ and $4$, and is equal to $4$ if and only if the $2$-plane is complex.  The detailed discussion of the Fubini--Study metric in terms of the moving frame can be found in \cite{ref_Chern}*{section 8}.

Let $\{z_\mu\}$ be the complex coordinate for the fibers of $(T^*\BCP^n)^{(1,0)}$.  Then, the Fubini--Study metric induces the following metric on $T^*\BCP^n$.
\begin{align} \label{hk_metric0}
\sum_{\mu=1}^n \left( |\ta^\mu|^2 + |\dd z_\mu - z_\nu\,\ta_\mu^\nu|^2 \right) ~.
\end{align}
The complex structure of $\BCP^n$ induces a complex structure on $T^*\BCP^n$, with respect to which $\ta^\mu$ and $\dd z_\mu - z_\mu\,\ta_\mu^\nu$ are $(1,0)$-forms.

\begin{rmk}
We briefly explain the convention of the correspondence between real and complex moving frames.  Write $\ta^\mu$ as $\ul{\om}^{2\mu-1} + \ii\,\ul{\om}^{2\mu}$.  Then, $\{\ul{\om}^{2\mu-1}\}_{\mu=1}^n\cup\{\ul{\om}^{2\mu}\}_{\mu=1}^n$ constitutes an orthonormal coframe.  Let $\{\ul{e}_{2\mu-1}\}_{\mu=1}^{n}\cup\{\ul{e}_{2\mu}\}_{\mu=1}^n$ be the dual orthonormal frame.  They satisfy $J(\ul{e}_{2\mu-1}) = \ul{e}_{2\mu}$, where $J$ is the complex structure as an endomorphism on the (real) tangent bundle.  Denote by $\ul{\om}_A^B$ the connection $1$-forms, where $1\leq A,B\leq 2n$.  Namely,
$\nabla\ul{e}_A = \ul{\om}_A^B\ot\ul{e}_B$.  Since $J$ is parallel, $\ul{\om}_{2\nu}^{2\mu} = \ul{\om}_{2\nu-1}^{2\mu-1}$ and $\ul{\om}_{2\nu}^{2\mu-1} = - \ul{\om}_{2\nu-1}^{2\mu}$.  The Hermitian connection $\ta_\nu^\mu$ is equal to
\begin{align} \label{hk_conn0}
\ul{\om}_{2\nu-1}^{2\mu-1} + \ii\,\ul{\om}_{2\nu-1}^{2\mu} = \ul{\om}_{2\nu}^{2\mu} - \ii\,\ul{\om}_{2\nu}^{2\mu-1} ~.
\end{align}

For the total space of the cotangent bundle, write $z_\mu$ as $x_\mu - \ii\,y_\mu$.  Under the (real) isomorphism
\begin{align*} \begin{array}{ccl}
(T^*\BCP)^{(1,0)} & \cong & T^*\BCP^n \quad \text{(real cotangent bundle)} \\
z_\mu\,\ta^\mu & \leftrightarrow & \re(z_\mu\,\ta^\mu) = x_\mu\,\ul{\om}^{2\mu-1} + y_\mu\,\ul{\om}^{2\mu} ~,
\end{array} \end{align*}
the metric \eqref{hk_metric0} is equal to
\begin{align*}
\sum_{\mu=1}^n \left( (\ul{\om}^{2\mu-1})^2 + (\ul{\om}^{2\mu})^2 + (\dd x_\mu - x_\nu\,\ul{\om}_{2\mu-1}^{2\nu-1} - y_\nu\,\ul{\om}_{2\mu-1}^{2\nu})^2 + (\dd y_\mu - x_\nu\,\ul{\om}_{2\mu}^{2\nu-1} - y_\nu\,\ul{\om}_{2\mu}^{2\nu})^2 \right) ~.
\end{align*}
\end{rmk}

\subsubsection{Spherical coordinate for the fibers}
Let $r = \sqrt{\sum_\mu |z_\mu|^2}$ be the distance to the zero section with respect to the metric \eqref{hk_metric0}.  Again, the exterior derivative of $r$ and the tautological $1$-form are two naturally defined $1$-forms on $T^*\BCP^n$.  In terms of the complex coordinate, they read
\begin{align*}
\re\left( \frac{1}{r}\br{z}_\mu\,(\dd z_\mu - z_\nu\,\ta^\nu_\mu) \right)  \qquad\text{and}\qquad
\re\left( z_\mu\,\ta^\mu \right) ~,
\end{align*}
respectively.  Their images under the complex structure give another two $1$-forms, which are the imaginary parts of the above two $(1,0)$-forms multiplied by $-1$.

With this understood, consider the following complex version of the spherical change of gauge.  It means an extension of $\frac{1}{r}(z_1,\cdots,z_n)$ to a unitary frame for $\BC^n$, which can be done on any simply connected open subset of $\BC^n\backslash\{0\}$.  More precisely, choose a smooth map $T_\nu^\mu$ from a simply connected open subset of $\BC^n\backslash\{0\}$ to $\RU(n)$ such that
\begin{align*}
T_\mu^1(z) = \frac{1}{r} z_\mu \qquad\text{for}\quad\mu=1,2,\ldots,n \quad\text{and}\quad z \in \text{Domain}(T_\mu^\nu)\subset\BC^n\backslash\{0\}.
\end{align*}
It follows that the following $1$-forms also constitute a unitary coframe of type $(1,0)$ for \eqref{hk_metric0}
\begin{align} \begin{split}
\sm^\mu &= T_\nu^\mu\,{\ta}^\nu ~, \\
\sm^{n+\mu} &= (T^{-1})^\nu_{\mu}\,(\dd z_\nu - z_\gm\,{\ta}_\nu^\gm) = \ol{T_\nu^\mu}\,(\dd z_\nu - z_\gm\,{\ta}_\nu^\gm)
\end{split} \label{hk_frame0} \end{align}
where $\mu\in\{1,\ldots,n\}$.

\subsubsection{The Calabi metric}

In terms of this coframe, the Calabi metric is of the following form
\begin{align} \label{hk_metric1}
\left| c(r)\,\sm^1 \right|^2 + \sum_{j=2}^n \left| b(r)\,\sm^j \right|^2 + \left( h(r)\,\dd r \right)^2 + \left( \frac{f(r)}{r}\,\im\sm^{n+1}\right)^2 + \sum_{j=2}^n \left| \frac{a(r)}{r}\,\sm^{n+j} \right|^2 ~,
\end{align}
where
\begin{align} \begin{split}
a &= \sinh(r) ~, \\
b&= \cosh(r) ~,
\end{split} &\begin{split}
c &= h = \sqrt{\cosh(2r)} ~, \\
f&= \frac{1}{2}\frac{\sinh(2r)}{\sqrt{\cosh(2r)}} ~.
\end{split} \label{hk_coeff0} \end{align}

In \cite{ref_DS}, Dancer and Swann found an easy way to construct the hyper-K\"ahler metric.  They wrote down the ansatz for the three hyper-K\"ahler forms, and imposed the $\dd$-closed condition.  One of the K\"ahler forms is
\begin{align} \label{hk_kform1}
\frac{hf}{r}\,\dd r\w\im{\sm^{n+1}} + \frac{\ii}{2}c^2\,\sm^1\w\ol{\sm^1} + \frac{\ii}{2}b^2\sum_{j=2}^n\sm^j\w\ol{\sm^j} + \frac{\ii}{2}\frac{a^2}{r^2}\sum_{j=2}^n\sm^{n+j}\w\ol{\sm^{n+j}} ~,
\end{align}
and the other two are the imaginary and real parts of
\begin{align} \label{hk_kform2}
hc\,\dd r\w\sm^1 + \frac{\ii fc}{r}\,(\im{\sm^{n+1}})\w\sm^1 - \frac{ab}{r}\sum_{j=2}^n\sm^j\w\sm^{n+j} ~.
\end{align}
If \eqref{hk_kform1} and \eqref{hk_kform2} are annihilated by the exterior derivative, the coefficient functions must obey
\begin{align} \begin{split}
&\frac{\dd a^2}{\dd r} = 2hf = \frac{\dd b^2}{\dd r} ~, \quad \frac{\dd c^2}{\dd r} = 4hf ~, \quad a^2 + b^2 = c^2 ~, \\
& \frac{\dd(ab)}{\dd r} = hc ~,\quad \frac{\dd(fc)}{\dd r} = hc ~, \quad\text{and}\quad  ab = fc ~.
\end{split} \label{hk_coeff1} \end{align}
It is a straightforward computation to check that \eqref{hk_coeff0} does solve \eqref{hk_coeff1}.  One can consult \cite{ref_CGLP1}*{section 4} for the discussion on solving \eqref{hk_coeff1}.

For the connection and curvature computation in the following subsubsections, it is more transparent to simplify the expressions by using the hyper-K\"ahler equation \eqref{hk_coeff1} than by plugging in the explicit solution \eqref{hk_coeff0}.  Note that \eqref{hk_coeff1} implies that $b^2 - a^2$ is a constant, which is $1$ for the explicit solution \eqref{hk_coeff0}.

\subsubsection{Connection $1$-forms}
To compute the connection $1$-forms of the metric, it is easier to choose a complex structure and take a unitary frame.  With respect to the complex structure corresponding to the K\"ahler form \eqref{hk_kform1}, we have the unitary coframe:
\begin{align} \label{hk_frame1}
\xi^1 &= c\,\sm^1 ~,  &\xi^j &= b\,\sm^j ~,  &\xi^{n+1} &= h\,\dd r + \frac{\ii f}{r}\im\sm^{n+1} ~,  &\xi^{n+j} &= \frac{a}{r}\,\sm^{n+j}
\end{align}
for $j\in\{2,\ldots,n\}$.  Denote by $\xi_\nu^\mu$ the Hermitian connection $1$-forms.  They can be found by the structure equation \eqref{hk_conn1}:
\begin{align} \begin{split}
\xi_1^1 &= \ii\left(\frac{2f}{c^2} - \frac{1}{f}\right)\im\xi^{n+1} = -\xi_{n+1}^{n+1} ~, \\
\xi_{n+1}^1 &= \frac{2f}{c^2}\,\xi^1 ~, \qquad\qquad \xi_{n+k}^j = 0 ~,
\end{split} \begin{split}
\xi^j_{n+1} &= \frac{f}{b^2}\,\xi^j = \xi^1_{n+j} ~, \\
\xi_{n+1}^{n+j} &= \frac{f}{a^2}\,\xi^{n+j} = -\xi_j^1 ~.
\end{split} \label{hk_conn2} \end{align}
The hyper-K\"ahler equation \eqref{hk_coeff1} is used to simplify the above expressions.  The components $\xi_j^k$ and $\xi_{n+j}^{n+k}$ are related to the connection of the Fubini--Study metric as follows:
\begin{align} \label{hk_conn3}
\xi_j^k &= -\xi_{n+k}^{n+j} = T^k_\mu\,\ta^\mu_\nu\,\ol{T^j_\nu} - (\dd T_\mu^k)\,\ol{T_\mu^j} + \ii\frac{f}{b^2}\delta_j^k\,\im\xi^{n+1} ~.
\end{align}

\subsubsection{Riemann curvature tensor} \label{sc_Calabi_curv}
In terms of the unitary coframe \eqref{hk_frame1}, the curvatures are as follows:
\begin{align*}
\FR_1^1 &= \frac{2}{c^6} \left( \xi^1\w\ol{\xi^1} - \xi^{n+1}\w\ol{\xi^{n+1}} \right) + \frac{1}{c^4} \left( \xi^j\w\ol{\xi^j} - \xi^{n+j}\w\ol{\xi^{n+j}} \right) = -\FR_{n+1}^{n+1} ~, \\
\FR_1^j &= \frac{1}{c^4} \left( \xi^j\w\ol{\xi^1} - \xi^{n+1}\w\ol{\xi^{n+j}} \right) = -\FR_{n+j}^{n+1} ~, \\
\FR_1^{n+j} &= -\frac{1}{c^4} \left( \xi^{n+j}\w\ol{\xi^1} + \xi^{n+1}\w\ol{\xi^j} \right) = \FR_j^{n+1} ~, \\
\FR_1^{n+1} &= -\frac{2}{c^6}\,\xi^{n+1}\w\ol{\xi^1} ~, \\
\FR_j^{n+k} &= -\frac{1}{c^2} \left( \xi^{n+j}\w\ol{\xi^k} + \xi^{n+k}\w\ol{\xi^j} \right) ~, \\
\FR_j^k &= \frac{1}{c^2} \left( \xi^k\w\ol{\xi^j} - \xi^{n+j}\w\ol{\xi^{n+k}} \right) \\
&\quad + \delta_j^k \left( \frac{1}{c^4} \left( \xi^1\w\ol{\xi^1} - \xi^{n+1}\w\ol{\xi^{n+1}} \right) + \frac{1}{c^2} \left( \xi^i\w\ol{\xi^i} - \xi^{n+i}\w\ol{\xi^{n+i}} \right) \right) = -\FR_{n+k}^{n+j} ~.
\end{align*}
The equality between different curvature components is a consequence of the hyper-K\"ahler geometry.

\section{Geometry of the Bryant--Salamon metrics}\label{sc_BS}
In \cite{ref_BS}, Bryant and Salamon constructed complete manifolds with special holonomy.  They constructed three examples with holonomy $\gt$, and one example with holonomy $\sps$, each of which is  the total space of a vector bundle.

\subsection{Bryant--Salamon manifolds}
This subsection is a brief review on the construction of Bryant and Salamon.  We first review the general framework of the metric construction on a vector bundle, and then specialize in their examples.

\subsubsection{Bundle construction}\label{sc_bc}
Let $(B^n,\ul{g})$ be a Riemannian manifold, and  $E\stackrel{\pi}{\to} B$ be a rank $m$ vector bundle.  Suppose that $E$ carries a bundle metric and a metric connection.  Then, these data naturally induce a Riemannian metric on the total space $E$.  The construction goes as follows.  With the connection, the tangent space of $E$ decomposes into vertical and horizontal subspaces.  These two subspaces are defined to be orthogonal to each other.  The metric on the vertical subspace is given by the original bundle metric; the metric on the horizontal subspace is the pull-back of the metric $\ul{g}$.

This metric can be seen explicitly in terms of the moving frame.  Take a local orthonormal coframe $\{\ul{\om}^j\}_{j=1}^n$ on an open subset $U\subset B$, and a local orthonormal basis of sections $\{\fs_\nu\}_{\nu=1}^m$ that trivializes $E|_U$.  Denote by $\nabla_A$ a metric connection for $E$.  Let $A^\mu_\nu$ be the the connection $1$-forms of $\nabla_A$ with respect to ${\fs_\nu}$, namely, $\nabla_A \fs_\nu = \sum_{\mu=1}^m A^\mu_\nu\,\fs_\mu$.  Thus, $[A_\nu^\mu]$ is an $\mathfrak{o}(m)$-valued $1$-form on $U$.  Let $\{y^\mu\}_{\mu=1}^m$ be the coordinate for the fibers of $E|_U$ induced by $\{\fs_\mu\}_{\mu=1}^m$.  The Riemannian metric $g_b$ on $E$ induced by the bundle metric is
\begin{align}
g_b=\sum_{j=1}^n (\ul{\om}^j)^2 + \sum_{\mu=1}^m\big(\dd y^\mu + A_\nu^\mu\,y^\nu\big)^2 ~.
\label{bundle_metric}\end{align}
Comparing \eqref{bundle_metric} with \eqref{metric0} and \eqref{hk_metric0}, the connection matrices
in section \ref{sc_Stenzel} and \ref{sc_Calabi} are defined with respect to the Levi-Civita connection on the tangent bundle, so the induced connection matrices on the cotangent bundle are the negative transposes in  \eqref{metric0} and \eqref{hk_metric0}, while the connection matrices here are on an arbitrary vector bundle.

In our discussion on the bundle construction, the indices $i,j,k$ are assumed to belong to $\{1,2,\ldots,n=\dim B\}$, and the indices $\mu,\nu,\gm,\sm$ are assumed to belong to $\{1,2,\ldots,m = \rank E\}$.

Here is a relation that will be used later.  The exterior derivative of $\dd y^\mu + A_\nu^\mu\,y^\nu$ can be written as $F_\nu^\mu\,y^\nu - A_\nu^\mu\w(\dd y^\nu + A_\gm^\nu\,y^\gm)$, where
\begin{align}
F_\nu^\mu = \dd A_\nu^\mu + A^\mu_\gm\w A^\gm_\nu = \oh F^\mu_{\nu\,ij}\,\ul{\om}^i\w\ul{\om}^j
\label{bd_cuv0} \end{align}
is the curvature of $\nabla_A$.

\subsubsection{Rescaling the metric}
Let $s = {\sum_{\mu}(y^\mu)^2}$ be the distance square to the zero section with respect to the Riemannian metric $g_b$ in \eqref{bundle_metric}.  For any two smooth, positive functions $\alpha(s),\beta(s)$ defined for $s\geq 0$,
\begin{align}
g_{\alpha,\beta} &= \sum_{j} \big(\alpha\,\ul{\om}^j\big)^2 + \sum_{\mu}\big(\beta\,(\dd y^\mu + A_\nu^\mu\,y^\nu)\big)^2
\label{bd_metric0} \end{align}
also defines a Riemannian metric on $E$.  Let
\begin{align}
\om^j = \alpha\,\ul{\om}^j   \qquad\text{and}\qquad
\om^{n+\mu} = \beta\,(\dd y^\mu + A_\nu^\mu\,y^\nu) ~.
\label{bd_frame0} \end{align}
It follows that
\begin{align}
\dd s = \frac{2}{\bt}y^\mu\,\om^{n+\mu} ~.
\end{align}
Note that $\{\om^j\}_{j=1\cdots n} \cup \{\om^{n+\mu}\}_{\mu=1\cdots m}$  form an orthonormal coframe of the metric \eqref{bd_metric0}.  Their exterior derivatives read
\begin{align} \begin{split}
\dd{\om}^j &= -\ul{\om}_i^j\w{\om}^i - \frac{2\alpha'}{\alpha\beta} y^\mu\,{\om}^j\w{\om}^{n+\mu} ~, \\
\dd{\om}^{n+\mu} &= \beta\,F^\mu_\nu\,y^\nu - A^\mu_\nu\w{\om}^{n+\nu} - \frac{2\beta'}{\beta^2}y^\nu\,{\om}^{n+\mu}\w{\om}^{n+\nu}
\end{split} \label{bd_dom} \end{align}
where $\ul{\om}_i^j$ is the connection $1$-form of the Levi-Civita connection of $(B,\ul{g})$.  By a direct computation, we find that the connection $1$-forms of the Levi-Civita connection of $(E,g_{\af,\bt})$ are
\begin{align}
{\om}_i^j &= \ul{\om}_i^j + \frac{\beta}{2\alpha^2} F^\mu_{\nu\,ij}\,y^\nu\,{\om}^{n+\mu} ~, \label{bd_conn0} \\
{\om}_i^{n+\mu} &= \frac{\beta}{2\alpha^2} F^\mu_{\nu\,ij}\,y^\nu\,{\om}^j - \frac{2\alpha'}{\alpha\beta}y^\mu\,{\om}^i ~, \label{bd_conn1} \\
{\om}_{n+\nu}^{n+\mu} &= A_\nu^\mu + \frac{2\beta'}{\beta^2}(y^\nu\,{\om}^{n+\mu} - y^\mu\,{\om}^{n+\nu}) ~. \label{bd_conn2}
\end{align}

The mixed component $\om^{n+\mu}_i$ is different from the other two components; it involves only the curvature but not the connection.  That is to say, $\om^{n+\mu}_i$ is a tensor:
\begin{align*}
\om^i\ot\om^{n+\mu}_i\ot\fs_\mu &= \bt\, F(\cdot\,,\cdot)(y^\mu\,\fs_\mu) - \frac{2\af \af'}{\bt}\,\ul{g}(\cdot\,,\cdot)(y^\mu\,\fs_\mu) : \CH\times\CH\to\CV
\end{align*}
where $\CH\cong\pi^*TB$ is the horizontal subspace, and $\CV\equiv\pi^*E$ is the vertical subspace over the total space of $E$.

\subsubsection{Examples of Bryant and Salamon}
For the examples of Bryant and Salamon, the base manifold $B$ is either a sphere or a complex projective space.  The metric $\ul{g}$ is the standard metric.  The vector bundle $E$ is constructed from the tangent bundle or the spinor bundle, and the metric connection is induced from the Levi-Civita connection.  In what follows, $\kp$ is the sectional curvature of the round metric when the base is the sphere, and is half the holomorphic sectional curvature of the Fubini-Study metric when the base is $\BCP^2$.

The first example \cite{ref_BS}*{p.840} is the spinor bundle over the $3$-sphere, $\BS(S^3)$.  The $\gt$ metric has
\begin{align}
\af(s) = (3\kp)^\oh (1 + s)^{\frac{1}{3}}   \qquad\text{and}\qquad
\bt(s) = 2 (1 + s)^{-\frac{1}{6}} ~.
\label{bd_metric1} \end{align}

The next two examples \cite{ref_BS}*{p.844} are the bundle of anti-self-dual $2$-forms over the $4$-sphere and the $2$-dimensional complex projective space, $\Lambda^2_-(S^4)$ and $\Lambda^2_-(\BCP^2)$.  They have
\begin{align}
\af(s) = (2\kp)^{\oh} (1 + s)^{\frac{1}{4}}   \qquad\text{and}\qquad
\bt(s) = (1 + s)^{-\frac{1}{4}} ~.
\label{bd_metric2} \end{align}

The last example \cite{ref_BS}*{p.847} is the spinor bundle of negative chirality over the $4$-sphere, $\BS_-(S^4)$.  The $\sps$ metric has
\begin{align}
\af(s) = (5\kp)^\oh (1 + s)^{\frac{3}{10}}   \qquad\text{and}\qquad
\bt(s) = 2 (1 + s)^{-\frac{1}{5}} ~.
\label{bd_metric3} \end{align}

We will refer $\BS(S^3)$, $\Lambda^2_-(S^4)$, $\Lambda^2_-(\BCP^2)$ and $\BS_-(S^4)$ with the Ricci flat metric as the \emph{Bryant--Salamon manifolds}.

The above coefficient functions are derived from the special holonomy equation, which is a first order elliptic system.  According to their paper, the special holonomy equation reduces to the following equation:
\begin{align}
\af' = \kp_1\,\frac{\bt^2}{\af} \qquad\text{and}\qquad  \bt' = -\kp_2\,\frac{\bt^3}{\af^2}
\label{bd_eq0} \end{align}
for some \emph{positive} constants $\kp_1$ and $\kp_2$.  For $\BS(S^3)$, $\kp_1 = \kp/4$ and $\kp_2 = \kp/8$.  For $\Lambda^2_-(S^4)$ and $\Lambda^2_-(\BCP^2)$, $\kp_1 = \kp_2 = \kp/2$.  For $\BS_-(S^4)$, $\kp_1 = 3\kp/8$ and $\kp_2 = \kp/4$.

One can easily construct some functional equations from \eqref{bd_eq0}.  Here are two relations that will be used later:
\begin{align}
\left(\frac{\af^2}{\bt^2}\right)' = 2(\kp_1 + \kp_2)  \quad&\Rightarrow\quad \frac{\af^2}{\bt^2} = \left(\frac{\af(0)}{\bt(0)}\right)^2 + 2(\kp_1 + \kp_2)s ~;  \label{bd_relation1} \\
\left(\frac{\bt}{\af^2}\right)' &= - (2\kp_1 + \kp_2)\,\frac{\bt^3}{\af^4} ~.
\label{bd_relation2} \end{align}


\section{The uniqueness of the zero section}\label{sc_rigidity}
The following lemma is about the rigidity of a compact minimal submanifold.  The authors believe it must be known to experts in the field. Due to the lack of a precise reference, the proof is  included for completeness.

\begin{lem}\label{lem_key}
Let $(M,{g})$ be a Riemannian manifold.  Suppose that $\psi$ is a smooth function on $M$ whose Hessian is non-negative definite.
Then, any {compact}, {minimal} submanifold of $M$ must be contained in the set where $\Hess_M\psi$ degenerates.  In addition, $\psi$ takes constant value on the submanifold if it is connected.
\end{lem}
\begin{proof}
Let $\Sm\subset M$ be a compact submanifold, and $p$ be a point in $\Sm$.  Choose an orthonormal frame $\{e_j\}$ for $T\Sm$ on some neighborhood of $p$ in $\Sm$.  Consider the trace of the Hessian of $\psi$ on $T\Sm$:
\begin{align*}
0 \leq  \tr_\Sigma \Hess \psi=\sum_j\Hess \psi(e_j,e_j) &= \sum_j \left( e_j(e_j(\psi)) - ({\nabla}_{e_j}e_j)(\psi) \right) \\
&= \Delta^\Sm\psi - H(\psi)
\end{align*}
where ${\nabla}$ is the covariant derivative of $(M,{g})$, $\Delta^\Sm$ is the Laplace--Beltrami operator of the induced metric on $\Sm$, and $H = \sum_{j}(\nabla_{e_j}e_j)^\perp$ is the mean curvature vector of $\Sm$.  When $\Sm$ is minimal, it implies that $\Delta^\Sm\psi\geq0$.  The lemma follows from the compactness of $\Sm$ and the maximum principle.
\end{proof}

\subsection{The Stenzel metric case}
In this subsection, we apply Lemma \ref{lem_key} to show that the zero section is the only compact, special Lagrangian submanifold in $T^*S^n$.

\begin{thm}\label{unique_Stenzel}
When $n>1$, any compact, minimal submanifold in $T^*S^n$ with the Stenzel metric must belong to the zero section.
\end{thm}

\begin{proof}
Consider the smooth function $\psi = \rho^2$, which is the square of the distance to the zero section (with respect to the Stenzel metric).  Due to Lemma \ref{lem_key}, it suffices to show that the Hessian of $\psi$ is positive definite outside the zero section. Since $\Hess(\rho)=\nabla \dd\rho$ and $\dd \rho=\om^{n+1}$, we compute
\begin{align*}
\Hess (\psi) &= 2\,\om^{n+1}\ot \om^{n+1} + 2\rho \nabla \om^{n+1} ~.
\end{align*}
By \eqref{conn0} and \eqref{Ricci1}, we obtain
\begin{align} \label{Hess0}
\Hess (\psi) &= 2\,\om^{n+1}\ot\om^{n+1} - 2\rho (n-1)C\,\om^1\ot\om^1 -2\rho A\sum_{j=2}^n\om^j\ot\om^j - 2\rho B\sum_{j=2}^n\om^{n+j}\ot\om^{n+j} ~.
\end{align}
According to Lemma \ref{lem_est1}, $\Hess(\psi)$ is positive definite when $r>0$ (equivalently, when $\rho>0$).  This finishes the proof of the theorem.
\end{proof}

\subsection{The Calabi metric case}
In this subsection, we prove that the zero section is the only compact, minimal submanifold in $T^*\BCP^n$.

\begin{thm}\label{unique_Calabi}
Any compact, minimal submanifold in $T^*\BCP^n$ with the Calabi metric must belong to the zero section.
\end{thm}

\begin{proof}
The function $\rho = \int_0^r\sqrt{\cosh(2u)}\dd u$ is the distance function to the zero section with respect to the Calabi metric and $\dd \rho=\re \xi^{n+1}$ by \eqref{hk_coeff0} and \eqref{hk_frame1}. Consider the smooth function $\psi = \rho^2$ and compute as in the last theorem,
\begin{align*}
\Hess (\psi) &= 2\,(\re \xi^{n+1})\ot (\re \xi^{n+1}) + 2\rho\,\re (\nabla \xi^{n+1}) ~.
\end{align*}
By \eqref{hk_conn2},
\begin{align*}
\re (\nabla \xi^{n+1}) &= \frac{2f}{c^2}\,|\xi^1|^2 + \frac{f}{b^2}\sum_{j=2}^n|\xi^j|^2 + (\frac{1}{f} -\frac{2f}{c^2})\,(\im\xi^{n+1})\ot(\im \xi^{n+1}) + \frac{f}{a^2}\sum_{j=2}^n|\xi^{n+j}|^2 ~.
\end{align*}
According \eqref{hk_coeff0}, it is not hard to see that $\Hess(\psi)$ is positive definite when $r>0$, and the theorem follows from Lemma \ref{lem_key}.
\end{proof}


\subsection{The Bryant--Salamon metric case }
In this subsection, we examine the uniqueness of the zero section as a minimal submanifold.  

\begin{lem}\label{lem_bd0}
Let $(B^n,\ul{g})$ be a Riemannian manifold.  Let $E\to B$ be a rank $m$ vector bundle with a bundle metric and a metric connection.  Denote by $s$ the square of the distance to the zero section with respect to the metric $g_b$ \eqref{bundle_metric} on $E$.
For any two smooth positive functions $\alpha(s)$ and $\beta(s)$, endow $E$ the Riemannian metric $g_{\alpha,\beta}$ defined by \eqref{bd_metric0}.  Then, the Hessian of $s$ with respect to $g_{\alpha, \beta}$ is positive definite outside the zero section if and only if $\alpha'>0$ and $\beta>2s|\beta'|$ for $s>0$.
\end{lem}

\begin{proof}
Suppose that $\psi$ is a smooth function on $E$ depending only on $s = \sum_{\mu}(y^\mu)^2$.  Its exterior derivative is
\begin{align}
\dd\psi &= \psi'\,\dd s = 2{\psi'} y^\mu\,\dd y^\mu = \frac{2\psi'}{\beta} y^\mu\,{\om}^{n+\mu} ~.
\label{bd_dr} \end{align}
Let $\{\bar{e}_j\}_{j=1}^n\cup\{\bar{e}_{n+\mu}\}_{\mu=1}^m$ be the frame dual to the coframe \eqref{bd_frame0}.  Since $\bar{e}_j(\psi) \equiv 0$, the Hessian of $\psi$ along $(\bar{e}_i,\bar{e}_j)$ is
\begin{align}
\Hess(\psi)(\bar{e}_i,\bar{e}_j) &= \bar{e}_i(\bar{e}_j(\psi)) - ({\nabla}_{\bar{e}_i}\bar{e}_j)(\psi)  \notag \\
&= \frac{4\alpha'\psi'}{\alpha\beta^2}s\,\delta_i^j + \frac{\psi'}{\alpha^2} y^\mu\,y^\nu\, F^\mu_{\nu\;ij} \notag \\
&= \frac{4\alpha'\psi'}{\alpha\beta^2}s\,\delta_i^j \label{bd_Hess0}
\end{align}
where the last equality uses the fact that $[F_\nu^\mu]$ is skew-symmetric in $\mu$ and $\nu$.  It is not hard to see that the Hessian of $\psi$ along $(\bar{e}_{n+\mu},\bar{e}_j)$ vanishes.  Along $(\bar{e}_{n+\mu},\bar{e}_{n+\nu})$,
\begin{align}
\Hess(\psi)(\bar{e}_{n+\mu},\bar{e}_{n+\nu}) &= \bar{e}_{n+\mu}(\bar{e}_{n+\nu}(\psi)) - ({\nabla}_{\bar{e}_{n+\mu}}\bar{e}_{n+\nu})(\psi) \notag \\
&= \left( \frac{2\psi'}{\beta^2}\,\delta_\mu^\nu + y^\mu\,y^\nu\,\frac{2}{\beta}\left(\frac{2\psi'}{\beta}\right)' \right) - \left(y^\mu\,y^\nu\frac{4\beta'\psi'}{\beta^3} - \frac{4\beta'\psi'}{\beta^3}s\,\delta_\mu^\nu\right) \notag \\
&= \left( \frac{2}{\beta^2} + \frac{4\beta'}{\beta^3}s \right)\psi'\,\delta_\mu^\nu + \left(\frac{4\psi'}{\beta^2}\right)'\,y^\mu\,y^\nu ~. \label{bd_Hess1}
\end{align}
Substituting $\psi = s$, the lemma follows from \eqref{bd_Hess0}, \eqref{bd_Hess1} and the fact that the eigenvalues of the matrix $[y^\mu\,y^\nu]$ are $s$ and $0$, where $0$ has geometric multiplicity $m-1$.
\end{proof}

Applying this lemma to the Bryant--Salamon manifolds leads to the following theorem.
\begin{thm}\label{unique_BS}
Any compact minimal submanifold of the Bryant--Salamon manifolds must be contained in the zero section.
\end{thm}

\begin{proof}
This follows directly from \eqref{bd_eq0}, \eqref{bd_relation1} and Lemma \ref{lem_bd0}.  Moreover, by \eqref{bd_relation1}, we have
\begin{align}\label{bd_Hess2}
\Hess(\psi) \geq 4\kp_1\frac{\bt}{\af^3}s\sum_{i=1}^n\om^i\ot\om^i + \frac{2}{\af^2}\left( \left(\frac{\af(0)}{\bt(0)}\right)^2 + 2\kp_1 s\right)\sum_{\mu=1}^m\om^{n+\mu}\ot\om^{n+\mu} ~.
\end{align}
This Hessian estimate will be needed later.
\end{proof}

\section{Further estimates needed for the stability theorem}\label{sc_estimate}

In this section, we begin preparations for the proof of Theorem \ref{stable}.  Each manifold in Assumption \ref{model_space} is the total space of a vector bundle $\pi:E\to B$.  The base $B$ naturally sits inside $E$ as the zero section.  In each case, we introduce a differential form $\Om$ with the properties explained in section \ref{sc_intro}, and calculate the covariant derivatives of $\Om$.  The calculations will be applied to derive $\CC^1$ estimate of the mean curvature flow.

\subsection{Estimates from linear-algebraic decomposition} \label{sc_linear}

Each metric in Assumption \ref{model_space} admits a local orthonormal coframe $\{\om^j\}_{j=1}^n\cup\{\om^{n+\mu}\}_{\mu=1}^m$ such that $\bigcap_{j=1}^n\ker\om^j = \pi^* E$.  As in section \ref{sc_bc}, The subbundle $\pi^*E\subset TE$ will be called the \emph{vertical subspace}, and will be denoted by $\CV$.  The orthogonal subbundle $\CH\subset TE$ is given by $\bigcap_{\mu=1}^m\ker\om^{n+\mu}$, and is isomorphic to $\pi^*TB$.  This bundle $\CH$ will be referred as the \emph{horizontal subspace}.

In terms of the frame, the $n$-form $\Om$ is $\om^1\w\om^2\w\cdots\w\om^n$.  Let $p\in E$, and suppose that $L\subset T_pE$ is an oriented $n$-dimensional subspace with $\Om(L)>0$.  Then, $L$ can be regarded as the graph of a linear map from $\CH_p$ to $\CV_p$.  By the singular value decomposition, there exist orthonormal bases $\{u_j\}_{j=1}^n$ for $\CH_p$, $\{v_\mu\}_{\mu=1}^m$ for $\CV_p$, and angles $\ta_j\in[0,\pi)$ such that
\begin{align} \label{linear0}
\{e_j = \cos\ta_j\,u_j + \sin\ta_j\,v_j\}_{j=1}^n \qquad\text{and}\qquad
\{e_{n+\mu} = -\sin\ta_\mu\,u_\mu + \cos\ta_\mu\,v_\mu\}_{\mu=1}^m
\end{align}
constitute orthonormal bases for $L$ and $L^\perp$, respectively.  For $j>m$, $v_j$ is set to be the zero vector, and $\ta_j$ is set to be zero.  For $\mu>n$, $u_\mu$ is set to be the zero vector, and $\ta_\mu$ is set to be zero.

Note that neither the frame $\{e_j\}\cup\{e_{n+\mu}\}$ nor $\{u_j\}\cup\{v_\mu\}$ is necessarily dual to $\{\om^j\}\cup\{\om^{n+\mu}\}$.  In any event, $[\om^j(u_i)]_{i,j}$ is an $n\times n$ (special) orthogonal matrix, and $[\om^{n+\mu}(v_\nu)]_{\nu,\mu}$ is a $m\times m$ orthogonal matrix.  Denote by
\begin{align} \label{linear2}
\fs = \max_j |\sin\theta_j| ~.
\end{align}
The following estimates are straightforward to come by:
\begin{align} \label{linear3}
\sum_{i=1}^n \left| (\om^j\ot\om^k)(e_i,e_i) \right| \leq n ~,\quad
\sum_{i=1}^n \left| (\om^{n+\mu}\ot\om^j)(e_i,e_i) \right| \leq n\fs
\end{align}
and
\begin{align}  \begin{split}
&\left| (\om^1\w \cdots \w \om^n)(e_{n+\mu}, e_1,\cdots ,\widehat{e_i}, \cdots, e_n) \right| \leq \fs ~, \\
&\left| (\om^1\w \cdots \w \om^n)(e_{n+\mu}, e_{n+\nu},  e_1, \cdots, \widehat{e_i}, \cdots, \widehat{e_j} \cdots, e_n) \right| \leq \fs^2 ~, \\
&\left| (\om^{n+\mu}\w \om^1 \w \cdots \w \widehat{\om^i} \w \cdots \w \om^n)(e_1,  \cdots, e_n) \right|\leq n \fs ~, \\
&\left| (\om^{n+\mu}\w \om^1 \w \cdots \w \widehat{\om^i} \w \cdots \w \om^n)(e_{n+\nu}, e_1,\cdots, \widehat{e_j} \cdots, e_n) \right| \leq 1 ~, \\
&\left| (\om^{n+\mu} \w \om^{n+\nu} \w \om^1\w \cdots \w\widehat{\om^i}\w \cdots\w \widehat{\om^j} \w \cdots \w \om^n)(e_1,\cdots, e_n) \right| \leq n(n-1)\fs^2 \\
\end{split} \label{linear1} \end{align}
for any $i,j,k \in \{1, \ldots, n\}$ and $\mu,\nu\in\{1, \ldots, m\}$.  To illustrate, we briefly explain the derivation of the first and third inequalities
in \eqref{linear1}.  By \eqref{linear0}, $(\om^1\w \cdots \w \om^n)(e_{n+\mu}, e_1,\cdots ,\widehat{e_i}, \cdots, e_n)$ vanishes unless $\mu = i$, and
\begin{align*}
&\left| (\om^1\w \cdots \w \om^n)(e_1,\cdots, e_{i-1},e_{n+i},e_{i+1}, \cdots, e_n) \right| \\
=\,& \left| (\om^1\w \cdots \w \om^n)(\cos\ta_1 u_1,\cdots ,\cos\ta_{i-1}u_{i-1},-\sin\ta_i u_i, \cos\ta_{i+1}u_{i+1}, \cdots, \cos\ta_n u_n) \right| \leq \fs ~.
\end{align*}
For the third one,
\begin{align*}
&\left| (\om^{n+\mu}\w \om^1 \w \cdots \w \widehat{\om^i} \w \cdots \w \om^n)(e_1,  \cdots, e_n) \right| \\
\leq\, &\sum_{k=1}^n \left| \om^{n+\mu}(e_k) \right|\,\left| (\om^1 \w \cdots \w \widehat{\om^i} \w \cdots \w \om^n)(e_1, \cdots, \widehat{e_k}, \cdots, e_n) \right| \leq n\fs ~.
\end{align*}

Suppose that $\Sm\subset E$ is an oriented, $n$-dimensional submanifold with $\Om(T_p\Sm)>0$.  Applying the above construction to $T_p\Sm$ gives a continuous function $\fs$ on $\Sm$.  With this understanding, the remainder of this section is devoted to estimating
\begin{align*}
\nabla_{e_j}\Om \quad\text{ and }\quad (\tr_{T_p\Sm}\nabla^2\Om)(T_p\Sm) = \sum_{j=1}^n(\nabla^2_{e_j,e_j}\Om)(e_1,\cdots,e_n)
\end{align*}
in terms of $\fs$ and the distance to the zero section.  These estimates will be used in the proof of Theorem \ref{stable}.

\subsection{The Stenzel metric case}\label{sc_Stenzel_est}

Consider the Stenzel metric on $T^*S^n$, and let
\begin{align}
\Om = \om^1\w\om^2\w\cdots\w\om^n ~,
\label{Stenzel_Om}\end{align}
where  $\om^1$ and $\om^j, j=2,\cdots, n$ are defined in  \eqref{frame1}.  The $n$-form $\Om$ is not parallel.  In order to establish the estimates on $\nabla\Om$ and $\nabla^2\Om$, it is convenient to introduce the following notations:
\begin{align} \begin{split}
\Phi &= \om^2\w\cdots\w\om^n ~, \\
\Phi^j &= \iota(e_j)\Phi = (-1)^j\,\om^2\w\cdots\w\widehat{\om^j}\w\cdots\w\om^n ~, \\
\Phi^{jk} &= \iota(e_k)\iota(e_j)\Phi = \begin{cases}
(-1)^{j+k}\,\om^2\w\cdots\w\widehat{\om^k}\w\cdots\w\widehat{\om^j}\w\cdots\w\om^n   &\text{if } k<j ~, \\
(-1)^{j+k+1}\,\om^2\w\cdots\w\widehat{\om^j}\w\cdots\w\widehat{\om^k}\w\cdots\w\om^n   &\text{if } k>j
\end{cases}
\end{split} \label{Phi0} \end{align}
for any $j,k\in\{2,\ldots,n\}$.

By \eqref{conn0}, \eqref{conn1} and \eqref{Ricci1}, the covariant derivatives of the coframe $1$-forms are as follows.
\begin{align} \begin{split}
\nabla\om^1 &= -B\,\om^{n+j}\ot\om^j + (n-1)C\,\om^1\ot\om^{n+1} + A\,\om^j\ot\om^{n+j} ~,  \\
\nabla\om^j &= -\om_k^j\ot\om^k + B\,\om^{n+j}\ot\om^1 + A\,\om^j\ot\om^{n+1} - C\,\om^1\ot\om^{n+j} ~,  \\
\nabla\om^{n+1} &= -(n-1)C\,\om^1\ot\om^1 - A\,\om^j\ot\om^j - B\,\om^{n+j}\ot\om^{n+j} ~, \\
\nabla\om^{n+j} &= -A\,\om^j\ot\om^1 + C\,\om^1\ot\om^j + B\,\om^{n+j}\ot\om^{n+1} - \om^j_k\ot\om^{n+k} ~.
\end{split} \label{cov1} \end{align}
We compute the covariant derivative of $\Phi$ and $\Phi^j$.
\begin{align}
\nabla\Phi &= (\nabla\om^j)\w\Phi^j \notag \\
&= (B\,\om^{n+j})\ot(\om^1\w\Phi^j) - (C\,\om^1)\ot(\om^{n+j}\w\Phi^j) + (A\,\om^j)\ot(\om^{n+1}\w\Phi^j) ~, \label{cov2} \\
\nabla\Phi^j &= (\nabla\om^k)\w\Phi^{jk} \notag \\
\begin{split}&= (B\,\om^{n+k})\ot(\om^1\w\Phi^{jk}) + (A\,\om^k)\ot(\om^{n+1}\w\Phi^{jk}) \\
&\quad - (C\,\om^1)\ot(\om^{n+k}\w\Phi^{jk}) + \om_j^k\ot\Phi^k ~. \label{cov3} \end{split}
\end{align}

Putting \eqref{cov1} and \eqref{cov2} together gives the covariant derivative of $\Om$.
\begin{align} \begin{split}
\nabla\Om &= (n-1)(C\,\om^1)\otimes(\om^{n+1}\w\Phi) + (A\,\om^j)\ot(\om^{n+j}\w\Phi) \\
&\quad -(C\,\om^1)\ot(\om^1\w\om^{n+j}\w\Phi^j) + (A\,\om^j)\ot(\om^1\w\om^{n+1}\w\Phi^j) ~.
\end{split} \label{cov4} \end{align}

{
We also compute the second covariant derivative of $\Om$ by computing the covariant derivative of the four terms on the right hand side of \eqref{cov4}.  By \eqref{cov1}, \eqref{cov2} and \eqref{cov3}:
\begin{align*}
\nabla^2\Omega = (n-1)\,\text{I} + \text{II} + \text{III} + \text{IV}
\end{align*}
where 
\begin{align*}
\text{I} &= \nabla\left( (C\,\om^1)\ot(\om^{n+1}\w\Phi) \right) \\
&= -(n-1)(C^2\, \om^1\ot\om^1)\otimes\Om - (BC\,\om^{n+j}\ot\om^1)\ot(\om^{n+j}\w\Phi + \om^1\w\om^{n+1}\w\Phi^j) \\
&\quad + \left( \dot{C}\,\om^{n+1}\ot\om^1 - BC\,\om^{n+j}\ot\om^j + (n-1)C^2\,\om^1\ot\om^{n+1} + AC\,\om^j\ot\om^{n+j} \right)\ot(\om^{n+1}\w\Phi) \\
&\quad - (C^2\,\om^1\ot\om^1)\ot(\om^{n+1}\w\om^{n+j}\w\Phi^j) ~,
\end{align*}
\begin{align*}
\text{II} &= \nabla\left( (A\,\om^j)\ot(\om^{n+j}\w\Phi) \right) \\
&= -(A^2\,\om^j\ot\om^j)\ot\Om + (AB\,\om^{n+j}\ot\om^j)\ot(\om^{n+1}\w\Phi) + (AB\,\om^{n+k}\ot\om^j)\ot(\om^{n+j}\w\om^1\w\Phi^k) \\
&\quad + \left( \dot{A}\,\om^{n+1}\ot\om^j + A^2\,\om^j\ot\om^{n+1} + AB\,\om^{n+j}\ot\om^1 - AC\,\om^1\ot\om^{n+j} \right)\ot(\om^{n+j}\w\Phi) \\
&\quad -(AC\,\om^1\ot\om^j)\ot(\om^{n+j}\w\om^{n+k}\w\Phi^k) + (A^2\,\om^k\ot\om^j)\ot(\om^{n+j}\w\om^{n+1}\w\Phi^k)~,
\end{align*}
\begin{align*}
\text{III} &= -\nabla\left( (C\,\om^1)\ot(\om^1\w\om^{n+j}\w\Phi^j) \right) \\
&= -(C^2\,\om^1\ot\om^1)\ot\Om - (C^2\,\om^1\ot\om^1)\ot\left( (n-1)\om^{n+1}\w\om^{n+j}\w\Phi^j - \om^1\w\om^{n+j}\w\om^{n+k}\w\Phi^{jk} \right) \\
&\quad -\left( \dot{C}\,\om^{n+1}\ot\om^1 - BC\,\om^{n+j}\ot\om^j + (n-1)C^2\,\om^1\ot\om^{n+1} + AC\,\om^j\ot\om^{n+j} \right)\ot(\om^1\w\om^{n+k}\w\Phi^k) \\
&\quad  -(AC\,\om^k\ot\om^1)\ot(\om^{n+k}\w\om^{n+j}\w\Phi^j + \om^1\w\om^{n+j}\w\om^{n+1}\w\Phi^{jk}) \\
&\quad - (BC\,\om^{n+j}\ot\om^1)\ot(\om^{n+j}\w\Phi + \om^1\w\om^{n+1}\w\Phi^j) ~,
\end{align*}
\begin{align*}
\text{IV} &= \nabla\left( (A\,\om^j)\ot(\om^1\w\om^{n+1}\w\Phi^j) \right) \\
&= -(A^2\,\om^j\ot\om^j)\ot\Om + (A^2\,\om^k\ot\om^j)\ot(\om^{n+k}\w\om^{n+1}\w\Phi^j) \\
&\quad +\left( \dot{A}\,\om^{n+1}\ot\om^j + A^2\,\om^j\ot\om^{n+1} + AB\,\om^{n+j}\ot\om^1 - AC\,\om^1\ot\om^{n+j} \right)\ot(\om^1\w\om^{n+1}\w\Phi^j) \\
&\quad +(AB\,\om^{n+j}\ot\om^j)\ot(\om^{n+1}\w\Phi) - (AB\,\om^{n+k}\ot\om^j)\ot(\om^1\w\om^{n+k}\w\Phi^j) \\
&\quad -(AC\,\om^1\ot\om^{j})\ot(\om^1\w\om^{n+1}\w\om^{n+k}\w\Phi^{jk}) ~.
\end{align*}

By examining the coefficient functions carefully, we conclude the following lemma.  Recall that the $\rho$ coordinate is the distance to the zero section.

\begin{lem} \label{lem_Om_est}
Consider $M = T^*S^n$ with the Stenzel metric \eqref{metric1}, and consider the $n$-form $\Om = \om^1\w\cdots\w\om^n$.  There exists a constant $K>0$ depending only on $n$ with the following property.  Suppose that $\Sm$ is a compact, oriented $n$-dimensional submanifold of $M$ such that $\rho(p)<1$ and $\Om(T_p\Sm)>0$ for any $p\in\Sm$.  Then,
\begin{align*}
&\left| \nabla_{X}\Om \right| (p) < K\rho(p)\,|X| \quad\text{ for any }X\in T_pM ~, \text{ and}\\
&-K\left( \rho^2(p) + \fs^2(p) \right) < (\tr_{T_p\Sm}\nabla^2\Om)(T_p\Sm)  < K\,\fs^2(p) - \frac{1}{K}\,\rho^2(p) ~.
\end{align*}
Here, $|\nabla_X\Om|$ means the metric norm of $\nabla_X\Om$ as a section of $(\Lambda^n T^*M)|_\Sm$ with respect to the metric induced by the Stenzel metric.
\end{lem}

\begin{proof}
In \eqref{cov4}, the coefficient functions consist of multiples of $A$ and $C$, which are of order $\rho$ by the second assertion of Lemma \ref{lem_est1}.

There are three types of terms in the formula of $\nabla^2\Om$.
\begin{enumerate}
\item The first type is in the direction of $\Om$. The first term of each of I, II, III, IV is of this type and their sum is
\begin{align*}
\left[-(n^2-2n+2)C^2\,\om^1\ot\om^1 - 2A^2\sum_{j=2}^n\om^j\ot\om^j \right]\otimes \Omega ~.
\end{align*}
The two-tensor in the bracket  is clearly semi-negative definite, and is of order $\rho^2$ by Lemma \ref{lem_est1}.

\item The second one is the linear combination of
$$ \eta = (\om^{n+i}\ot\om^{j})\ot(\om^{n+k}\w\om^1\w\cdots\w\widehat{\om^l}\w\cdots\w\om^n) \quad\text{and}\quad (\om^{i}\ot\om^{n+j})\ot(\text{same})$$
for $i,j,k,l\in\{1,,\ldots,n\}$.  By \eqref{linear3} and the third line of \eqref{linear1}, $ |(\tr_{T_p\Sm}\eta)(T_p\Sm)|\leq n^2\fs^2$.  On the other hand, the coefficient functions are constant multiples of $A^2, AB, BC, AC, C^2, \dot{A}$ or $\dot{C}$.  They are at most of order $1$ by Lemma \ref{lem_est1} and \ref{lem_est2}.

\item The third one is the linear combination of
$$ \eta = (\om^p\ot\om^q)\ot(\om^{n+i}\w\om^{n+j}\w\om^1\w\cdots\w\widehat{\om^k}\w\cdots\w\widehat{\om^l}\w\cdots\w\om^n) $$
for $i,j,k,l,p,q\in\{1,,\ldots,n\}$.  By \eqref{linear3} and the last line of \eqref{linear1}, $ |(\tr_{T_p\Sm}\eta)(T_p\Sm)|\leq n^3\fs^2$.  In these terms, the coefficient functions are multiples of $A^2, AC$ and $C^2$.  They are of order $\rho^2$ by Lemma \ref{lem_est1}.
\end{enumerate}
By the triangle inequality, the proof of this lemma is complete.
\end{proof}



\subsection{The Calabi metric case}
For the Calabi metric, we follow the notations introduced in section \ref{sc_Calabi} and consider the $2n$-form
\begin{align} \label{hk_Om}
\Om = k_n\,(\xi^1\w\xi^2\w\cdots\xi^n)\w\ol{(\xi^1\w\xi^2\w\cdots\w\xi^n)}
\end{align}
where $k_n = (-1)^{\frac{n(n-1)}{2}}(\frac{\ii}{2})^n$.  The $1$-forms $\xi^1$ and $\xi^j$ are given by \eqref{hk_frame1}.  The restriction of $\Om$ to the zero section coincides with the volume form of the zero section.  Let $\fu_\nu$ be the complexified tangent vector defined by $\xi^\mu(\fu_\nu) = \delta^\mu_\nu$ and $\ol{\xi^\mu}(\fu_\nu) = 0$.  Similar to the case of the Stenzel metric, let
\begin{align} \begin{split}
\Xi &= \xi^2\w\cdots\w\xi^n ~, \\
\Xi^j &= \iota(\fu_j)\Xi = (-1)^j\,\xi^2\w\cdots\w\widehat{\xi^j}\w\cdots\w\xi^n ~, \\
\Xi^{jk} &= \iota(\fu_k)\iota(\fu_j)\Xi = \begin{cases}
(-1)^{j+k}\,\xi^2\w\cdots\w\widehat{\xi^k}\w\cdots\w\widehat{\xi^j}\w\cdots\w\xi^n   &\text{if } k<j ~, \\
(-1)^{j+k+1}\,\xi^2\w\cdots\w\widehat{\xi^j}\w\cdots\w\widehat{\xi^k}\w\cdots\w\xi^n   &\text{if } k>j ~
\end{cases}
\end{split} \label{hk_Xi0} \end{align}
for any $j,k\in\{2,\ldots,n\}$.

By \eqref{hk_conn2} and \eqref{hk_conn3}, the covariant derivative of the unitary coframe reads:
\begin{align} \begin{split}
\nabla \xi^1 &= -\ii(\frac{2f}{c^2} - \frac{1}{f})\,(\im\xi^{n+1})\ot\xi^1 + \frac{b}{ac}\,\xi^{n+j}\ot\xi^j - \frac{2f}{c^2}\,\xi^1\ot\xi^{n+1} - \frac{a}{bc}\,\xi^j\ot\xi^{n+j} ~, \\
\nabla \xi^j &= -\frac{b}{ac}\,\ol{\xi^{n+j}}\ot\xi^1 - \xi^j_k\ot\xi^k - \frac{f}{b^2}\,\xi^j\ot\xi^{n+1} ~, \\
\nabla \xi^{n+1} &= \frac{2f}{c^2}\,\ol{\xi^1}\ot\xi^1 + \frac{f}{b^2}\,\ol{\xi^j}\ot\xi^j + \ii(\frac{2f}{c^2} - \frac{1}{f})\,(\im\xi^{n+1})\ot\xi^{n+1} + \frac{f}{a^2}\,\ol{\xi^{n+j}}\ot\xi^{n+j} ~, \\
\nabla \xi^{n+j} &= \frac{a}{bc}\,\ol{\xi^j}\ot\xi^1 - \frac{f}{a^2}\,\xi^{n+j}\ot\xi^{n+1} + \xi^k_j\ot\xi^{n+k} ~.
\end{split} \label{hk_cov1} \end{align}
It follows that the covariant derivative of $\Xi$ and $\Xi^j$ are as follows:
\begin{align}
\nabla \Xi &= (\nabla\xi^j)\w\Xi^j \notag \\
&= - \frac{b}{ac}\,\ol{\xi^{n+j}}\ot(\xi^1\w\Xi^j) - \xi_j^j\ot\Xi - \frac{f}{b^2}\,\xi^j\ot(\xi^{n+1}\w\Xi^j) ~,  \label{hk_cov2} \\
\nabla \Xi^j &= (\nabla\xi^k)\w\Xi^{jk} \notag \\
&= -\frac{b}{ac}\,\ol{\xi^{n+k}}\ot(\xi^1\w\Xi^{jk}) - \xi_k^k\ot\Xi^j + \xi^k_j\ot\Xi^k - \frac{f}{b^2}\,\xi^k\ot(\xi^{n+1}\w\Xi^{jk}) ~. \label{hk_cov3}
\end{align}


By combining \eqref{hk_cov1} and \eqref{hk_cov2}, the covariant derivative of $\xi^1\w\Xi = \xi^1\w\xi^2\w\cdots\w\xi^n$ is
\begin{align} \begin{split}
\nabla(\xi^1\w\Xi) &= \ii(\frac{1}{f} - \frac{2f}{c^2})\,(\im\xi^{n+1})\ot(\xi^1\w\Xi) - \frac{2f}{c^2}\,\xi^1\ot(\xi^{n+1}\w\Xi) \\
&\quad -\frac{a}{bc}\,\xi^j\ot(\xi^{n+j}\w\Xi) - \xi_j^j\ot(\xi^1\w\Xi) - \frac{f}{b^2}\,\xi^j\ot(\xi^1\w\xi^{n+1}\w\Xi^j) ~.
\end{split} \label{hk_cov4} \end{align}
Since $\Om$ is real,
\begin{align*}
\nabla\Om &= k_n\, \left(\nabla(\xi^1\w\Xi)\right)\w\ol{\xi^1}\w\ol{\Xi} + \text{(conjugate)} ~,
\end{align*}
so
\begin{align} \begin{split}
\nabla\Om &= - k_n \frac{2f}{c^2}\,\xi^1\ot\left( \xi^{n+1}\w\Xi\w\ol{\xi^1}\w\ol{\Xi} \right) - k_n \frac{a}{bc}\,\xi^j\ot\left( \xi^{n+j}\w\Xi\w\ol{\xi^1}\w\ol{\Xi} \right) \\
&\quad - k_n \frac{f}{b^2}\,\xi^j\ot\left( \xi^1\w\xi^{n+1}\w\Xi^j\w\ol{\xi^1}\w\ol{\Xi} \right) + (\text{their conjugates})
\end{split} \label{hk_cov5} \end{align}
where we have also used the fact that $\xi_j^k$ is skew-Hermitian, which is why the second-last term from the right hand side of  \eqref{hk_cov4} ends up canceling with its conjugate.


Note that by \eqref{hk_coeff1}, the second coefficient, $a/(bc)$, is equal to $f/b^2$.  The next step is to calculate the second order derivative of $\Om$, which is a sum of the covariant derivative of the six terms on the right hand side of \eqref{hk_cov5}.   Due to \eqref{hk_cov1}, \eqref{hk_cov2}, \eqref{hk_cov3}, \eqref{hk_cov4} and the relations \eqref{hk_coeff1}:
\begin{align}
\nabla^2\Om &= -k_n \left( (\text{I} +　\text{II} + \text{III}) + (\ol{\text{I}} + \ol{\text{II}} + \ol{\text{III}}) \right)
\label{hk_cov6} \end{align}
where
\begin{align*}
\text{I} &= \nabla\left(\frac{2f}{c^2}\,\xi^1\ot\Big( \xi^{n+1}\w\Xi\w\ol{\xi^1}\w\ol{\Xi} \Big)\right) \\
&= \frac{4f^2}{c^4}\,(\ol{\xi^1}\ot\xi^1)\ot\left( \xi^1\w\Xi\w\ol{\xi^1}\w\ol{\Xi} + \xi^{n+1}\w\Xi\w\ol{\xi^{n+1}}\w\ol{\Xi} \right) \\
&\quad + \left( \Big(\frac{2f}{c^2}\Big)'\frac{1}{h}\,(\re\xi^{n+1})\ot\xi^1 + \ii\frac{2f}{c^2}\Big(\frac{2f}{c^2} - \frac{1}{f}\Big)\,(\im\xi^{n+1})\ot\xi^1 - \frac{4f^2}{c^4}\,\xi^1\ot\xi^{n+1} \right. \\
&\qquad\quad \left. + \frac{2b^2}{c^4}\,\xi^{n+j}\ot\xi^j - \frac{2a^2}{c^4}\,\xi^j\ot\xi^{n+j} \right)\ot\left( \xi^{n+1}\w\Xi\w\ol{\xi^1}\w\ol{\Xi} \right) \\
&\quad + \frac{2b^2}{c^4}\,(\ol{\xi^{n+j}}\ot\xi^1)\ot\left(  \xi^{n+j}\w\Xi\w\ol{\xi^1}\w\ol{\Xi} \right) - \frac{2a^2}{c^4}\,(\ol{\xi^j}\ot\xi^1)\ot\left( \xi^{n+1}\w\Xi\w\ol{\xi^{n+j}}\w\ol{\Xi} \right) \\
&\quad + \frac{2b^2}{c^4}\,(\ol{\xi^{n+j}}\ot\xi^1)\ot\left( \xi^1\w\xi^{n+1}\w\Xi^j\w\ol{\xi^1}\w\ol{\Xi} \right) - \frac{2a^2}{c^4}\,(\ol{\xi^j}\ot\xi^1)\ot\left( \xi^{n+1}\w\Xi\w\ol{\xi^1}\w\ol{\xi^{n+1}}\w\ol{\Xi^j} \right) ~,
\end{align*}
\begin{align*}
\text{II} &= \nabla\left( \frac{f}{b^2}\,\xi^j\ot\Big( \xi^{n+j}\w\Xi\w\ol{\xi^1}\w\ol{\Xi}\Big) \right) \\
&= \frac{f^2}{b^4}\,(\ol{\xi^j}\ot\xi^j)\ot\left( \xi^1\w\Xi\w\ol{\xi^1}\w\ol{\Xi} \right) - \frac{1}{c^2}\,(\xi^{n+j}\ot\xi^j)\ot\left( \xi^{n+1}\w\Xi\w\ol{\xi^1}\w\ol{\Xi} \right) \\
&\quad + \left( \Big(\frac{f}{b^2}\Big)'\frac{1}{h}\,(\re\xi^{n+1})\ot\xi^j + \ii\frac{f}{b^2}\Big(\frac{2f}{c^2} - \frac{1}{f}\Big)\,(\im\xi^{n+1})\ot\xi^j - \frac{1}{c^2}\,\ol{\xi^{n+j}}\ot\xi^1 \right. \\
&\qquad\quad \left. - \frac{f^2}{b^4}\,\xi^j\ot\xi^{n+1} \right)\ot\left( \xi^{n+j}\w\Xi\w\ol{\xi^1}\w\ol{\Xi} \right) \\
&\quad -\frac{2a^2}{c^4}\,(\ol{\xi^1}\ot\xi^j)\ot\left( \xi^{n+j}\w\Xi\w\ol{\xi^{n+1}}\w\ol{\Xi} \right) - \frac{f^2}{b^4}\,(\ol{\xi^k}\ot\xi^j)\ot\left( \xi^{n+j}\w\Xi\w\ol{\xi^{n+k}}\w\ol{\Xi} \right) \\
&\quad + \frac{1}{c^2}\,(\ol{\xi^{n+k}}\ot\xi^j)\ot\left( \xi^1\w\xi^{n+j}\w\Xi^k\w\ol{\xi^1}\w\ol{\Xi} \right)  + \frac{f^2}{b^4}\,(\xi^k\ot\xi^j)\ot\left( \xi^{n+1}\w\xi^{n+j}\w\Xi^k\w\ol{\xi^1}\w\ol{\Xi} \right) \\
&\quad - \frac{f^2}{b^4}\,(\ol{\xi^k}\ot\xi^j)\ot\left( \xi^{n+j}\w\Xi\w\ol{\xi^1}\w\ol{\xi^{n+1}}\w\ol{\Xi^k} \right) ~,
\end{align*}
\begin{align*}
\text{III} &= \nabla\left( \frac{f}{b^2}\,\xi^j\ot\Big(\xi^1\w\xi^{n+1}\w\Xi^j\w\ol{\xi^1}\w\ol{\Xi}\Big) \right) \\
&= \frac{f^2}{b^4}\,(\ol{\xi^j}\ot\xi^j)\ot\left( \xi^1\w\Xi\w\ol{\xi^1}\w\ol{\Xi} \right) - \frac{1}{c^2}\,(\xi^{n+j}\ot\xi^j)\ot\left( \xi^{n+1}\w\Xi\w\ol{\xi^1}\w\ol{\Xi} \right) \\
&\quad + \left( \Big(\frac{f}{b^2}\Big)'\frac{1}{h}\,(\re\xi^{n+1})\ot\xi^j + \ii\frac{f}{b^2}\Big(\frac{2f}{c^2} - \frac{1}{f}\Big)\,(\im\xi^{n+1})\ot\xi^j - \frac{1}{c^2}\,\ol{\xi^{n+j}}\ot\xi^1 \right. \\
&\qquad\quad \left. - \frac{f^2}{b^4}\,\xi^j\ot\xi^{n+1} \right)\ot\left( \xi^1\w\xi^{n+1}\w\Xi^j\w\ol{\xi^1}\w\ol{\Xi} \right) \\
&\quad + \frac{f^2}{b^4}\,(\xi^k\ot\xi^j)\ot\left( \xi^{n+1}\w\xi^{n+k}\w\Xi^j\w\ol{\xi^1}\w\ol{\Xi} \right) + \frac{1}{c^2}\,(\ol{\xi^{n+k}}\ot\xi^j)\ot\left( \xi^1\w\xi^{n+k}\w\Xi^j\w\ol{\xi^1}\w\ol{\Xi} \right) \\
&\quad - \frac{2a^2}{c^4}\,(\ol{\xi^1}\ot\xi^j)\ot\left( \xi^1\w\xi^{n+1}\w\Xi^j\w\ol{\xi^{n+1}}\w\ol{\Xi} \right) - \frac{f^2}{b^4}\,(\ol{\xi^k}\ot\xi^j)\ot\left( \xi^1\w\xi^{n+1}\w\Xi^j\w\ol{\xi^{n+k}}\w\ol{\Xi} \right) \\
&\quad - \frac{f^2}{b^4}\,(\ol{\xi^k}\ot\xi^j)\ot\left( \xi^1\w\xi^{n+1}\w\Xi^j\w\ol{\xi^1}\w\ol{\xi^{n+1}}\w\ol{\Xi^k} \right) ~.
\end{align*}

Recall that the distance to the zero section with respect to the Calabi metric is $\rho = \int_0^r\sqrt{\cosh(2u)}\dd u$, and the asymptotic behavior of the coefficient functions near $\rho = 0$ can be found easily from \eqref{hk_coeff0}.  By applying \eqref{linear3} and \eqref{linear1} on \eqref{hk_cov5}and \eqref{hk_cov6}, a completely parallel argument as that in the proof of Lemma \ref{lem_Om_est} leads to the following lemma.

\begin{lem} \label{lem_Om_est_hk}
Consider $M=T^*\BCP^n$ with the Calabi metric \eqref{hk_metric1}, and consider the $2n$-form $\Om$ defined by \eqref{hk_Om}.  There exists a constant $K>0$ depending only on $n$ which has the following property.  Suppose that $\Sm$ is a compact, oriented $2n$-dimensional submanifold of $M$ such that $\rho(p)<1$ and $\Om(T_p\Sm)>0$ for any $p\in\Sm$.  Then,
\begin{align*}
&\left| \nabla_{X}\Om \right| < K\rho(p)\,|X| \quad\text{ for any }X\in T_pM ~, \text{ and}\\
&-K\left( \rho^2(p) + \fs^2(p) \right) < (\tr_{T_p\Sm}\nabla^2\Om)(T_p\Sm)  < K\,\fs^2(p) - \frac{1}{K}\,\rho^2(p) ~.
\end{align*}
\end{lem}

\subsection{The Bryant--Salamon metric case}
Consider the $n$-form
\begin{align}
\Om = \om^1\w\om^2\w\cdots\w\om^n ~.
\label{bd_Om} \end{align}
The notations in this subsection follow those in section \ref{sc_BS}, and the $1$-forms $\om^j$ are defined by \eqref{bd_frame0}.  Similar to the case of the Stenzel metric, it is convenient to introduce the following shorthand notations:
\begin{align}
\Om^j &= \iota(e_j)\Om = (-1)^{j+1}\om^1\w\cdots\w\widehat{\om^j}\w\cdots\w\om^n ~, \label{bd_Om1} \\
\Om^{jk} &= \iota(e_k)\iota(e_j)\Om = \begin{cases}
(-1)^{j+k}\,\om^1\w\cdots\w\widehat{\om^k}\w\cdots\w\widehat{\om^j}\w\cdots\w\om^n   &\text{if } k<j ~, \\
(-1)^{j+k+1}\,\om^1\w\cdots\w\widehat{\om^j}\w\cdots\w\widehat{\om^k}\w\cdots\w\om^n   &\text{if } k>j ~.
\end{cases} \label{bd_Om2}
\end{align}

The covariant derivative of $\Om$ is
\begin{align}
\nabla\Om &= (\nabla\om^j)\w\Om^j = \om^{n+\mu}_j\ot(\om^{n+\mu}\w\Om^j) ~.
\label{bd_cov1} \end{align}

For $\nabla^2\Om$, one expects that the covariant derivative of the curvature, $\nabla_A F\in\CC^\infty(B;T^*B\ot(\Lambda^2 T^*B \ot\End E))$, will show up.  The coefficient $1$-form of $\nabla_A F$ is
\begin{align}
(\nabla_A F)_{\nu\,jk}^\mu &= \dd F^\mu_{\nu\,jk} + A^\mu_\gm\,F^\gm_{\nu\,jk} - F^\mu_{\gm\,jk}\,A^\gm_\nu - F^\mu_{\nu\,ik}\,\ul{\om}^i_j - F^\mu_{\nu\,ji}\,\ul{\om}^i_k ~.
\label{bd_cov2} \end{align}
We are now ready to calculate the second order derivative of $\Om$.  Since
\begin{align}
\nabla \om^{n+\mu} &= -\om^{n+\mu}_i\ot\om^i - \om^{n+\mu}_{n+\nu}\ot\om^{n+\nu} \quad\text{and}\label{bd_cov3} \\
\nabla \Om^j &= (\nabla\om^k)\w\Om^{jk} = \om_j^k\ot\Om^k + \om_k^{n+\nu}\ot(\om^{n+\nu}\w\Om^{jk}) ~, \label{bd_cov4}
\end{align}
the covariant derivative of \eqref{bd_cov1} is
\begin{align} \begin{split}
\nabla^2\Om &= -(\om^{n+\mu}_i\ot\om^{n+\mu}_i)\ot\Om + (\om^{n+\nu}_k\ot\om^{n+\mu}_j)\ot(\om^{n+\mu}\w\om^{n+\nu}\w\Om^{jk}) \\
&\quad + \left( \nabla\om^{n+\mu}_j + \om^{n+\mu}_{n+\nu}\ot\om^{n+\nu}_j + \om^j_k\ot\om^{n+\mu}_k \right)\ot(\om^{n+\mu}\w\Om^j) ~.
\end{split} \label{bd_cov5} \end{align}
The first two coefficients on the right hand side of \eqref{bd_cov5} can be substituted by \eqref{bd_conn1}.  We compute the coefficient in the second line of \eqref{bd_cov5}.  By \eqref{bd_conn1},
\begin{align*}
\nabla\om^{n+\mu}_j &=\nabla\left( \frac{\bt}{2\af^2}F^\mu_{\nu\,jk}y^\nu\,\om^k \right) - \nabla\left( \frac{2\af'}{\af\bt}y^\mu\,\om^j \right) ~.
\end{align*}
With the help of \eqref{bd_frame0}, \eqref{bd_conn0}, \eqref{bd_dr} and \eqref{bd_cov3}, we have
\begin{align} \begin{split}
\nabla\left( \frac{\bt}{2\af^2}F^\mu_{\nu\,jk}y^\nu\,\om^k \right)
&= \frac{2}{\bt}\left( \frac{\bt}{2\af^2} \right)' F^\mu_{\nu\,jk} y^\nu y^\gm\,\om^{n+\gm}\ot\om^k + \frac{1}{2\af^2} F^\mu_{\nu\,jk}\,\om^{n+\nu}\ot\om^k\\
&\quad + \frac{\bt}{2\af^2} F^\mu_{\nu\,jk} y^\nu\,\om^{n+\gm}_k\ot\om^{n+\gm} - \frac{\bt^2}{4\af^2} (F^\mu_{\nu\,jk}y^\nu)(F^\sm_{\gm\,ik}y^\gm) \,\om^{n+\sm}\ot\om^i \\
&\quad + \frac{\bt}{2\af^2}y^\nu \left( \dd F^\mu_{\nu\,jk} - F^\mu_{\gm\,jk}\,A^\gm_\nu - F^\mu_{\nu\,ji}\,\ul{\om}^i_k \right)\ot\om^k ~,
\end{split} \label{bd_cov6} \end{align}
and
\begin{align} \begin{split}
-\nabla\left( \frac{2\af'}{\af\bt}y^\mu\,\om^j \right)
&= -\frac{2}{\bt} \left(\frac{2\af'}{\af\bt}\right)' y^\mu y^\nu\,\om^{n+\nu}\ot\om^j - \frac{2\af'}{\af\bt^2}\,\om^{n+\mu}\ot\om^j \\
&\quad + \frac{\af'}{\af^2} y^\mu F^\gm_{\nu\;kj} y^\nu\,\om^{n+\gm}\ot\om^k - \frac{2\af'}{\af\bt}y^\mu\,\om^{n+\nu}_j\ot\om^{n+\nu} \\
&\quad + \frac{2\af'}{\af\bt} y^\gm\,A^\mu_\gm\ot\om^j + \frac{2\af'}{\af\bt}y^\mu\,\ul{\om}^j_k\ot\om^k ~.
\end{split} \label{bd_cov7} \end{align}
Due to \eqref{bd_conn0} \eqref{bd_conn1} and \eqref{bd_conn2},
\begin{align} \begin{split}
\om^{n+\mu}_{n+\nu}\ot\om^{n+\nu}_j &= \frac{2\bt'}{\bt^2}(y^\nu\,\om^{n+\mu} - y^\mu\,\om^{n+\nu})\ot\om_j^{n+\nu} \\
&\quad - \frac{2\af'}{\af\bt} y^\nu\, A^\mu_\nu\ot\om^j + \frac{\bt}{2\af^2} y^\nu (A^\mu_\gm\,F^\gm_{\nu\,jk})\ot\om^k ~,
\end{split} \label{bd_cov8} \end{align}
and
\begin{align} \begin{split}
-\om^k_j\ot\om^{n+\mu}_k &= - \frac{\bt}{2\af^2} F^\gm_{\nu\,jk}y^\nu\,\om^{n+\gm}\ot\om^{n+\mu}_k \\
&\quad + \frac{2\af'}{\af\bt}y^\mu\,\ul{\om}^k_j\ot\om^k - \frac{\bt^2}{2\af^2} y^\nu (F^\mu_{\nu\,ik}\,\ul{\om}^i_j)\ot\om^k ~.
\end{split} \label{bd_cov9} \end{align}
To summarize the above computations, note that
\begin{itemize}
\item the sum of the last terms of \eqref{bd_cov6}, \eqref{bd_cov8} and \eqref{bd_cov9} is  a multiple of $\nabla_A F$  by \eqref{bd_cov2};
\item the last term of \eqref{bd_cov7} cancels with the second-last term of \eqref{bd_cov9};
\item the second-last term of \eqref{bd_cov7} cancels with that of \eqref{bd_cov8}.
\end{itemize}

The above computation is for a general bundle construction.  We now examine the expressions for the Bryant-Salamon metric.  It is more convenient to consider the function $s = \sum_\mu (y^\mu)^2$, which is equivalent to the distance square to the zero section on any compact region.

\begin{lem} \label{lem_Om_est_BS}
Consider the $n$-form $\Om = \om^1\w\cdots\w\om^n$ on each of the Bryant--Salamon manifolds.  There exists a constant $K>0$ which has the following property.  Suppose that $\Sm$ is a compact, oriented $n$-dimensional submanifold of $M$ such that $s(p)<1$ and $\Om(T_p\Sm)>0$ for any $p\in\Sm$.  Then,
\begin{align*}
&\left| \nabla_{X}\Om \right| < K \sqrt{s(p)}\, |X| \quad\text{ for any }X\in T_pM ~, \text{ and}\\
&-K\left( s(p) + \fs^2(p) \right) < (\tr_{T_p\Sm}\nabla^2\Om)(T_p\Sm)  < K\,\fs^2(p) - \frac{1}{K}s(p) ~.
\end{align*}
\end{lem}

\begin{proof}
The coefficient functions $\af$ and $\bt$ have explicit expressions, \eqref{bd_metric1}, \eqref{bd_metric2} and \eqref{bd_metric3}.  The only property needed here is that $\af$, $\bt$ and their derivatives are uniformly bounded when $s\in[0,1]$.  The estimate on $\nabla_X\Om$ follows directly from \eqref{bd_conn1} and \eqref{bd_cov1}.

To estimate $\nabla^2\Om$, consider \eqref{bd_cov5}, \eqref{bd_cov6}, \eqref{bd_cov7}, \eqref{bd_cov8} and \eqref{bd_cov9}:
\begin{itemize}
\item The first coefficient $2$-tensor on the right hand side of \eqref{bd_cov5} is non-positive definite, and is of order $s$ when $s(p)<1$.
\item The second term on the right hand side of \eqref{bd_cov5} carries $\om^{n+\mu}\w\om^{n+\nu}\w\cdots$.
\item As explained in Appendix \ref{sc_bd_Fcurv}, $\nabla_A F\equiv0$.  Thus, each term of the third coefficient $2$-tensor on the right hand side of \eqref{bd_cov5} carries at least one $\om^{n+\mu}$-codirection.
\end{itemize}
Then, the lemma follows from \eqref{linear3} and \eqref{linear1}.
\end{proof}

\section{The stability of zero sections}\label{sc_stability}

Suppose $\Sigma_t$ is a mean curvature flow of $n$-dimensional compact submanifolds in an ambient Riemannian manifold $M$ and $\Omega$ is an $n$-form on $M$. For any
point $p\in \Sigma_t$, let $\{e_1,\cdots,e_n\}$ be an orthonormal
frame of $T\Sigma_t$ near $p$ and $\{e_{n+1}, \cdots, e_{n+m}\}$
be an orthonormal frame of the normal bundle of $\Sigma_t$ near
$p$. In the following, the indexes $i, j, k$ range from $1$ to $n$, the indexes $\alpha, \beta, \gamma$ range from $n+1$ to $n+m$, and repeated indexes are summed.  Let $h_{\alpha ij}=\langle\nabla_{e_i} e_j , e_\alpha\rangle$ denote the coefficients of the second fundamental form of $\Sigma_t$.   Here, $\nabla$ is the Levi-Civita connection of the ambient manifold $M$.

We first recall the following proposition from \cite{ref_W4}*{Proposition 3.1}
\begin{prop} Along the mean curvature flow $\Sigma_t$ in $M$, 
$*\Omega =\Omega(e_1, \cdots, e_n)$ satisfies
\begin{align}\label{*Omega}
\begin{split}
\frac{\dd}{\dd t} *\Omega &=\Delta^{\Sm_t} *\Omega +*\Omega
(\sum_{\alpha ,i,k}h_{\alpha i k}^2)\\
&-2\sum_{\alpha, \beta, k}[\Omega_{\alpha \beta 3\cdots n}
h_{\alpha 1k}h_{\beta 2k} + \Omega_{\alpha 2 \beta\cdots n}
h_{\alpha 1k}h_{\beta 3k} +\cdots +\Omega_{1\cdots (n-2) \alpha
\beta} h_{\alpha (n-1)k}
h_{\beta nk}]\\
&-2(\nabla_{e_k}\Omega) (e_\alpha, \cdots, e_{n})h_{\alpha 1 k}
-\cdots-2(\nabla_{e_k}\Omega) (e_{1}, \cdots, e_\alpha
)h_{\alpha n k}\\
&-\sum_{\alpha, k}[\Omega_{\alpha 2 \cdots n}R_{\alpha kk1}
+\cdots +\Omega_{1\cdots (n-1)\alpha}R_{\alpha kkn}]-(\nabla^2_{e_k,e_k}\Omega)(e_{1},
\cdots,e_{n})
\end{split}
\end{align}
where $\Delta^{\Sm_t}$ denotes the time-dependent Laplacian on $\Sigma_t$, $\Omega_{\alpha\beta 3\cdots n}=\Omega(e_\alpha, e_\beta, e_3, \cdots, e_n)$ etc., and $R_{\alpha kk1}=R(e_\alpha, e_k, e_k, e_1)$, etc.\ are the coefficients of the curvature operators of $M$. 
\end{prop}
When $\Omega$ is a parallel form in $M$, $\nabla\Omega \equiv 0 $, this recovers an important formula in proving the long time existence result of the graphical mean 
curvature flow in \cite{ref_W3}.

\begin{rmk}
In \cite{ref_W4}, the frame $\{e_k\}_{k=1}^n$ is a geodesic frame at some $p\in M$, i.e.\ $\nabla^\Sm_{e_j}e_i$ vanishes at $p$.  Thus, the last term of \eqref{*Omega} is
$$ \left(\nabla^2_{e_k,e_k} \Om\right)(e_1,\cdots,e_n) = -(\nabla_{e_k}\nabla_{e_k}\Om)(e_1,\cdots,e_n) + (\nabla_H\Om)(e_1,\cdots,e_n)  $$
at $p\in M$.  This is exactly the formula in \cite{ref_W4}*{Proposition 3.1}.
\end{rmk}

\subsection{Proof of Theorem \ref{stable} for the Stenzel metric}
\begin{proof} We deal with the Stenzel metric first. Let $\epsilon$ be the constant to be determined and $\Sigma$ be a compact submanifold of $M$ that satisfies the assumption \eqref{condition0}. Throughout the proof, $K_i, i=0, 1, 2, \cdots$ denotes a positive constant that depends only on the dimension $n$. 
Denote by $\Sigma_t$ the mean curvature flow in $M$ with $\Sigma$ as the initial data.

We first prove the $\CC^0$ estimate.  As in the proof of Theorem \ref{unique_Stenzel}, let $\psi$ be the distance square to the zero section with respect to the Stenzel metric, or the square of the $\rho$ coordinate.  Its evolution equation along the mean curvature flow $\Sigma_t$ reads (see the proof of Theorem C in \cite{ref_W2})
\begin{align} \label{t_psi}
\frac{\dd}{\dd t} \psi &= \Delta^{\Sm_t} \psi - \tr_{\Sm_t}  \Hess \psi ~,
\end{align}
where
$\tr_{\Sm_t} \Hess \psi$ is the trace of the Hessian of $\psi$ over $\Sm_t$ and is always non-negative by Theorem \ref{unique_Stenzel}.  By the maximum principle, the maximum of $\psi$ on $\Sigma_t$ is non-increasing, and thus $\Sigma_t$ remains close to the zero section along the flow.

\

Next, we derive the $\CC^1$ estimate which amounts to showing that $*\Om=\Om(T_p\Sm_t)$ remains close to one. We claim that there exists a small enough $\ep\in(0,\oh)$ and a large enough $K_0>1$, both depending only on the dimension $n$, such that if the inequality $*\Om-K_0\psi>1-\ep$ holds on the initial data, then it remains true on $\Sigma_t$ at any subsequent time as long as the flow exists smoothly.

For the argument of contradiction, suppose $*\Omega-K_0\psi >1-\epsilon$ holds initially, and $*\Omega-K_0\psi =1-\epsilon$ for the first time at $T_0$. Our goal is to apply equation \eqref{*Omega} and the estimates in section \ref{sc_Stenzel_est} to show that for certain $\epsilon$ and $K_0$,
\begin{align} \label{*Omega3}
\frac{\dd}{\dd t}(*\Om - K_0\psi) \geq \Delta^{\Sm_t}(*\Om-K_0\psi) + \frac{1}{2}(*\Om-K_0\psi) |\CA|^2 ~.
\end{align}
where $|\CA|^2= \sum_{\alpha ,i,k}h_{\alpha i k}^2$.   Therefore, the minimum of $*\Omega-K_0\psi $ is non-decreasing in the interval $[0, T_0)$ and $*\Om - K_0\psi $ is indeed strictly greater than $1-\epsilon$ at $T_0$.

By \eqref{*Omega} and \eqref{t_psi}, the evolution equation of $*\Om - K_0\psi$ is:
\begin{align} \begin{split}
&\frac{\dd}{\dd t}(*\Om - K_0\psi) - \Delta^{\Sm_t}(*\Om - K_0\psi) \\
=\, &*\Omega
(\sum_{\alpha ,i,k}h_{\alpha i k}^2)-2\sum_{\alpha, \beta, k}[\Omega_{\alpha \beta 3\cdots n}
h_{\alpha 1k}h_{\beta 2k} + \Omega_{\alpha 2 \beta\cdots n}
h_{\alpha 1k}h_{\beta 3k} +\cdots +\Omega_{1\cdots (n-2) \alpha
\beta} h_{\alpha (n-1)k}
h_{\beta nk}]\\
& - 2(\nabla_{e_k}\Omega) (e_\alpha, \cdots, e_{n})h_{\alpha 1 k}
-\cdots-2(\nabla_{e_k}\Omega) (e_{1}, \cdots, e_\alpha
)h_{\alpha n k}\\
& - \sum_{\alpha, k}[\Omega_{\alpha 2 \cdots n}R_{\alpha kk1}
+\cdots +\Omega_{1\cdots (n-1)\alpha}R_{\alpha kkn}]-(\nabla^2_{e_k,e_k}\Omega)(e_{1},
\cdots,e_{n})+K_0\tr_{\Sigma_t}  \Hess \psi ~.
\end{split} \label{*Omega2} \end{align}
We aim at using the first and last term on the right hand side of \eqref{*Omega2} to control the rest of the terms.  At any $p\in\Sigma_t$, $T_p\Sm$ and $(T_p\Sm)^\perp$ have the following orthonormal bases constructed in section \ref{sc_linear}:
\begin{align*}
\{e_j = \cos\ta_j u_j + \sin\ta_j v_j\} \quad\text{and}\quad \{e_{n+j} = -\sin\ta_j u_j + \cos\ta_j v_j\} ~,
\end{align*}
respectively.  Recall that $\{ u_j, v_j\}_{j=1,\cdots n, }$ is an orthonormal basis for $T_p M$ such that $\om^i(v_j) = \om^{n+i}(u_j)=0$ for any $i, j$, and $[\om^i(u_j)]_{i,j}$ and $[\om^{n+i}(v_j)]_{i,j}$ are both $n\times n$ orthogonal matrices.  As in \eqref{linear2}, set $\fs$ to be $\max_j\{|\sin\ta_j|\}$.

It follows from $*\Om - K_0\psi > 1-\ep$ that $\psi = \rho^2 < \ep/K_0 < 1$.  In what follows, the point $p$ is always assumed to be at distance less than $1$ from the zero section.  It also follows from $*\Om - K_0\psi > 1-\ep$ that $*\Om = \Om(e_1,\cdots,e_n) = \prod_{j=1}^n\cos\ta_j > 1-\ep$.  Hence, $\cos\ta_j > 1-\ep > \oh$ and $\sin^2\ta_j < 2\ep$ for any $j\in\{1,\ldots,n\}$, and $\fs < 2 \ep^\oh$.

We now analyze the terms on the right hand side of \eqref{*Omega2}
\begin{enumerate}
\item By \eqref{Hess0} and Lemma \ref{lem_est1}, $\Hess\psi\geq \frac{1}{K_1}\sum_{j=1}^n(\rho^2\,\om^j\ot\om^j + \om^{n+j}\ot\om^{n+j})$ for some constant $K_1>0$.  Since $\cos\ta_j > \oh$ and $\sum_{j=1}^n\sin^2\ta_j\geq\max_j\{\sin^2\ta_j\} = \fs^2$,
\begin{align}\label{Hess}
\tr_{\Sm_t}\Hess \psi &\geq \frac{1}{K_1}(\rho^2\sum_{j=1}^n\cos^2\ta_j + \sum_{j=1}^n\sin^2\ta_j) > \frac{1}{K_1}(\frac{1}{4}\rho^2 + \fs^2) ~.
\end{align}

\item With the second line of \eqref{linear1} and the Cauchy--Schwarz inequality, $|\Omega_{\af\bt 3\cdots n}h_{\af 1k}h_{\bt 2k}|$ is bounded by $\fs^2|\CA|^2$.  Thus,
\begin{align*}
\left|2\sum_{\af, \bt, k} (\Om_{\af \bt 3\cdots n}
h_{\af 1k}h_{\bt 2k} + \cdots + \Om_{1\cdots (n-2) \af \bt} h_{\af (n-1)k}h_{\bt nk}) \right| \leq K_2\,\fs^2|\CA|^2
\end{align*}
for some constant $K_2 > 0$.

\item Due to the first assertion of Lemma \ref{lem_Om_est},
\begin{align*}
2\left| (\nabla_{e_k} \Om) (e_\af, \cdots, e_{n})h_{\af 1 k} + \cdots + (\nabla_{e_k} \Om) (e_{1}, \cdots, e_\af)h_{\af n k} \right| < K_3 \rho |\CA| \leq K_3^2\rho^2 + \frac{1}{4}|\CA|^2
\end{align*}
for some constant $K_3>0$.

\item According to the curvature computation in section \ref{sc_Stenzel_cruv}, the Riemann curvature tensor of the Stenzel metric satisfies $R(u_i,v_j,v_k,v_l) = 0 = R(v_i,u_j,u_k,u_l)$ for any $i,j,k,l\in\{1,\ldots,n\}$.  It follows that
\begin{align}
|R(e_{n+j},e_k,e_k,e_i)|\leq K_4(|\sin\ta_j| + |\sin\ta_k| + |\sin\ta_i|) \leq 3K_4\fs
\label{est_mixed} \end{align}
for some constant $K_4>0$.  This together with the first line of \eqref{linear1} implies that
\begin{align*}
\left| \sum_{\af, k}(\Om_{\af 2 \cdots n}R_{\af kk1} + \cdots + \Om_{1\cdots (n-1)\af}R_{\af kkn}) \right| \leq 3n^2K_4\,\fs^2 ~.
\end{align*}

\item By Lemma \ref{lem_Om_est}, there exists a constant $K_5>0$ such that
\begin{align*}
-\sum_{k=1}^n (\nabla^2_{e_k,e_k}\Om)(e_1,\cdots,e_n) &\geq -K_5\fs^2 ~.
\end{align*}
\end{enumerate}
It follows that the right hand side of \eqref{*Omega2} is greater than
\begin{align*}
*\Om|\CA|^2 + \frac{K_0}{K_1}(\frac{1}{4}\rho^2 + \fs^2) - K_2\fs^2|\CA|^2 - K_3^2\rho^2 - \frac{1}{4}|\CA|^2 - (3n^2K_4 + K_5)\,\fs^2 ~.
\end{align*}
It is clear that by taking $\ep$ to be sufficiently small and $K_0$ to be sufficiently large, the expression is greater than $\oh*\Om|\CA|^2 \geq \oh(*\Om - K_0\psi)|\CA|^2$.  This proves the differential inequality \eqref{*Omega3}.

\

Finally, we prove the $\CC^2$ estimate, which amounts to bounding the norm of the second fundamental form from above. The evolution equation for the norm of the second fundamental
form for a mean curvature flow in a general Riemannian manifold is derived in  \cite{ref_W1}*{Proposition 7.1}.
In particular, $|\CA|^2 = \sum_{\af, i, k} h_{\af i k}^2$ satisfies the following
equation along the flow:
\begin{align} \label{|A|^2} \begin{split}
\frac{\dd}{\dd t}|\CA|^2
&=\Delta^{\Sm_t} |\CA|^2 -2|\nabla^{\Sm_t} \CA|^2
+2[(\nabla_{e_k} R)_{\alpha ijk}
+(\nabla_{e_j} R)_{\alpha kik}]h_{\alpha ij}\\
&-4R_{lijk}h_{\alpha lk}h_{\alpha ij}
+8R_{\alpha \beta jk}h_{\beta ik}h_{\alpha ij}
-4R_{lkik}h_{\alpha lj}h_{\alpha ij}
+2R_{\alpha k\beta k}h_{\beta ij}h_{\alpha ij}\\
&+2\sum_{\alpha,\gamma, i,m}
(\sum_k h_{\alpha ik}h_{\gamma mk}
-h_{\alpha mk}h_{\gamma ik})^2
+2\sum_{i,j,m,k}(\sum_{\alpha} h_{\alpha ij}
h_{\alpha mk})^2 ~.
\end{split} \end{align}

In particular, we have
\begin{align*}
\frac{\dd}{\dd t}|\CA|^2 &\leq \Delta^{\Sm_t} |\CA|^2 - 2|\nabla^{\Sm_t} \CA|^2 + K_6|\CA|^4 + K_7 |\CA|^2 + K_8 |\CA| \\
&\leq \Delta^{\Sm_t} |\CA|^2 - 2|\nabla^{\Sm_t} \CA|^2 + K_6|\CA|^4 + (K_7 + \oh{K_8} )|\CA|^2 + \oh{K_8} ~,
\end{align*}
where $K_6, K_7$ and $K_8$ are positive constant that only depend on the geometry of $M$. Combining this with equation \eqref{*Omega2} 
and applying the method in \cite{ref_W3}*{p.540--542} yield the boundedness of the second fundamental form. The term $K_6|A|^4$ on the right
hand side, which can potentially lead to the finite time blow-up of $|\CA|^2$, is countered by the term $\frac{1}{2}(*\Om-K_0\psi) |\CA|^2$ on the right hand side of \eqref{*Omega3}.  Standard estimates for second order quasilinear parabolic system imply all higher derivatives are bounded, and we can apply Simon's convergence theorem \cite{ref_Simon} to conclude the smooth convergence as $t\rightarrow \infty$. In fact, in this case the smooth convergence can be proved directly by considering the derivatives of the second fundamental form. It follows from \eqref{Hess} and \eqref{t_psi} that $\psi|_{\Sm_t}$ converges exponentially to zero.
Similarly, it follows from \eqref{*Omega3} that $(1-*\Om + K_0\psi)|_{\Sm_t}$ converges to zero, and thus $*\Om$ converges to $1$.

\end{proof}

\subsection{Proof of Theorem \ref{stable} for the other cases}

The proofs for the Calabi metric and the Bryant--Salamon metric are almost the same as above.  We just highlight where it needs to be modified.  For the Calabi metric, the function $\psi$ is taken to be the distance square to the zero section, $(\int_0^r\sqrt{\cosh(2u)}\dd u)^2$.  For the Bryant--Salamon metric, the function $\psi$ is taken to be $s = \sum_{\mu}(y^\mu)^2$.

First of all, due to Theorem \ref{unique_Calabi} and \ref{unique_BS}, the $\CC^0$ estimate follows from the same argument.  Moreover, as can be seen in their proofs, there exists a constant $K_8>0$ such that
\begin{align*}
\Hess(\psi) > \frac{1}{K_8}(\psi + \fs^2)
\end{align*}
provided $\psi\leq1$.  This is item (i) in the proof of the $\CC^1$ estimate.  For the rest of the items,
\begin{itemize}
\item one simply has to replace Lemma \ref{lem_Om_est} by Lemma \ref{lem_Om_est_hk} and \ref{lem_Om_est_BS}, respectively;
\item according to the curvature computation in section \ref{sc_Calabi_curv} and appendix \ref{sc_bd_Rcurv}, their Riemann curvatures also admit the property that $R(\CH,\CV,\CV,\CV) = 0 = R(\CV,\CH,\CH,\CH)$.
\end{itemize}

For $\CC^2$ and all higher derivative estimates, the argument is completely the same.

\begin{appendix}

\section{The curvature of Bryant--Salamon manifolds}

This appendix is a brief summary on the curvature properties of the Bryant--Salamon manifolds.  For these bundle manifolds, the isometry group of the base acts transitively, and the action lifts to the total space of the vector bundle isometrically.  Hence, it suffices to examine the curvature at a particular fiber.

Fix a point $p$ in the base, and let $\ul{\om}^j$ be a geodesic frame at $p$.  Thus, at $p$, the Levi-Civita connection forms of the base metric vanish, $\ul{\om}_i^j|_p = 0$.  Since the bundle is either the spinor bundle or the bundle of anti-self-dual $2$-forms,  $\{\ul{\om}^j\}$ induces an orthonormal trivialization $\{\fs_\nu\}$ of the bundle by representation theory.  In other words, the bundle connection $A_\nu^\mu$ is a linear combination of $\ul{\om}_i^j$, and also vanishes at $p$.  It follows that $(\dd\ul{\om}_i^j)|_p = \ul{\FR}_i^j|_p$ and $(\dd A_\nu^\mu)|_p = F_\nu^\mu|_p$.

\subsection{The bundle curvature of the connection}\label{sc_bd_Fcurv}
In order to use \eqref{bd_conn0}, \eqref{bd_conn1} and \eqref{bd_conn2} to compute the curvature of \eqref{bd_metric0} over the fiber at $p$, we also need to know the exterior derivative of $F^\mu_{\nu\,ij}$.  They are locally defined functions on the base.  Since the metric here is the round metric or the Fubini-Study metric, these locally defined functions are actually locally constants.  The readers are directed to \cite{ref_BS} for the detail of the curvature computation.  We simply write down the answer.

For $\BS(S^3)$,
\begin{align}
F &= \frac{\kp}{2}\,\begin{bmatrix}
0 &-\ul{\om}^2\w\ul{\om}^3 &\ul{\om}^1\w\ul{\om}^3 &-\ul{\om}^1\w\ul{\om}^2 \\
\ul{\om}^2\w\ul{\om}^3 &0 &\ul{\om}^1\w\ul{\om}^2 &\ul{\om}^1\w\ul{\om}^3 \\
-\ul{\om}^1\w\ul{\om}^3 &-\ul{\om}^1\w\ul{\om}^2 &0 &\ul{\om}^2\w\ul{\om}^3 \\
\ul{\om}^1\w\ul{\om}^2 &-\ul{\om}^1\w\ul{\om}^3 &-\ul{\om}^2\w\ul{\om}^3 &0
\end{bmatrix} ~.
\label{bd_curv1} \end{align}
It is understood as an endomorphism-valued $2$-form with respect to the trivialization $\{\fs_\nu\}$.  The components of the curvature can be read off from the matrix.  For instance, $F^2_{4\,13} = \kp/2$.

For $\Lambda^2_-(S^4)$ and $\Lambda^2_-(\BCP^2)$,
\begin{align}
F&= \kp\,\begin{bmatrix}
0 &-\ul{\om}^1\w\ul{\om}^4 + \ul{\om}^2\w\ul{\om}^3 &\ul{\om}^1\w\ul{\om}^3 + \ul{\om}^2\w\ul{\om}^4 \\
\ul{\om}^1\w\ul{\om}^4 - \ul{\om}^2\w\ul{\om}^3 &0 &-\ul{\om}^1\w\ul{\om}^2 + \ul{\om}^3\w\ul{\om}^4 \\
-\ul{\om}^1\w\ul{\om}^3 - \ul{\om}^2\w\ul{\om}^4 &\ul{\om}^1\w\ul{\om}^2 - \ul{\om}^3\w\ul{\om}^4 & 0
\end{bmatrix} ~.
\label{bd_curv2} \end{align}
For $\BS_-(S^4)$,
\begin{align}
F&= \frac{\kp}{2}\,\begin{bmatrix}
0 &\ul{\om}^1\w\ul{\om}^2 - \ul{\om}^3\w\ul{\om}^4 &\ul{\om}^1\w\ul{\om}^3 + \ul{\om}^2\w\ul{\om}^4 & \ul{\om}^1\w\ul{\om}^4 - \ul{\om}^2\w\ul{\om}^3 \\
-\ul{\om}^1\w\ul{\om}^2 + \ul{\om}^3\w\ul{\om}^4&0 &-\ul{\om}^1\w\ul{\om}^4 + \ul{\om}^2\w\ul{\om}^3 &\ul{\om}^1\w\ul{\om}^3 + \ul{\om}^2\w\ul{\om}^4 \\
-\ul{\om}^1\w\ul{\om}^3 - \ul{\om}^2\w\ul{\om}^4&\ul{\om}^1\w\ul{\om}^4 - \ul{\om}^2\w\ul{\om}^3 &0 &-\ul{\om}^1\w\ul{\om}^2 + \ul{\om}^3\w\ul{\om}^4 \\
-\ul{\om}^1\w\ul{\om}^4 + \ul{\om}^2\w\ul{\om}^3&-\ul{\om}^1\w\ul{\om}^3 - \ul{\om}^2\w\ul{\om}^4 &\ul{\om}^1\w\ul{\om}^2 - \ul{\om}^3\w\ul{\om}^4 & 0
\end{bmatrix} ~.
\label{bd_curv3} \end{align}

By \eqref{bd_cov2}, it is clear that $\nabla_A F$ vanishes at $p$.  Since the argument applies to any $p\in B$, $\nabla_A F \equiv0$.

\subsection{The Riemann curvature tensor of the bundle metric}\label{sc_bd_Rcurv}
We are now ready to compute the curvature of Bryant--Salamon metric \eqref{bd_metric0} over the fiber at $p$.  The notation $|_p$ is abused to indicate the restriction to the fiber at $p$ in the following calculations.  The Levi-Civita connection of \eqref{bd_metric0} is discussed in section \ref{sc_BS}.

By \eqref{bd_conn0}, \eqref{bd_dom}, \eqref{bd_eq0} and \eqref{bd_relation2},
\begin{align*}
(\dd\om_i^j)|_p &= \ul{\FR}_i^j - \frac{\bt^2}{4\af^4}y^\nu F^\nu_{\mu\,ij}F^\mu_{\gm\,kl}y^\gm\,\om^k\w\om^l + \frac{1}{2\af^2}F^\mu_{\nu\,ij}\,\om^{n+\nu}\w\om^{n+\mu}  \\
&\quad - \frac{2\bt^2}{\af^4}(\kp_1+\kp_2) F^\mu_{\nu\,ij} y^\nu y^\gm\,\om^{n+\gm}\w\om^{n+\mu} ~, \\
(-\om_i^k\w\om_k^j)|_p &= \frac{\bt^2}{4\af^4} F^\mu_{\nu\,ik}y^\nu\,\om^{n+\mu}\w F^\eta_{\gm\,jk}y^\gm\,\om^{n+\eta} ~, \\
(-\om_i^{n+\mu}\w\om_{n+\mu}^j)|_p &= \frac{4\bt^2}{\af^4}\kp_1^2s\,\om^i\w\om^j - \frac{\bt^2}{4\af^4}y^\nu F^\nu_{\mu\,ik} F^\mu_{\gm\,jl}y^\gm\,\om^k\w\om^l ~.
\end{align*}
Thus, by \eqref{R_curv2}, we have
\begin{align}
\begin{split}  R_{jikl} &= \frac{1}{\af^2} \ul{R}_{jikl} - \frac{4\bt^2}{\af^4}\kp_1^2s\,(\dt_{jk}\dt_{il} - \dt_{jl}\dt_{ik}) \\
&\quad - \frac{\bt^2}{4\af^4} y^\nu (2F^\nu_{\mu\,ij}F^\mu_{\gm\,kl} + F^\nu_{\mu\,ik}F^\mu_{\gm\,jl} - F^\nu_{\mu\,il}F^\mu_{\gm\,jk}) y^\gm ~, \end{split} \label{bd_curv4}\\
\begin{split}  R_{ji(n+\mu)(n+\nu)} &= -\frac{1}{\af^2}F^\mu_{\nu\,ij} + \frac{2\bt^2}{\af^4}(\kp_1+\kp_2)(y^\nu F^\mu_{\gm\,ij} - y^\mu F^\nu_{\gm\,ij})y^\gm \\
&\quad + \frac{\bt^2}{4\af^4} y^\gm(F^\mu_{\gm\,ik}F^\nu_{\eta\,jk} - F^\nu_{\gm\,ik}F^\mu_{\eta\,jk})y^\eta ~.  \end{split} \label{bd_curv5}
\end{align}
In the above expression, $\ul{R}_{jikl} = R(\ul{e}_j,\ul{e}_i,\ul{e}_k,\ul{e}_l)$ is the Riemann curvature tensor of $(B,\ul{g})$, where $\{\ul{e}_j\}$ is the dual frame of $\{\ul{\om}^j\}$.

For the curvature component $\FR_i^{n+\mu}$,
\begin{align*}
(\dd\om_i^{n+\mu})|_p &= -\frac{2}{\af^2}\kp_1\,\om^{n+\mu}\w\om^i + \frac{1}{2\af^2}F^\mu_{\nu\,ij}\,\om^{n+\nu}\w\om^j + \frac{4\bt^2}{\af^4}\kp_1(\kp_1+\kp_2)y^\mu y^\gm\,\om^{n+\gm}\w\om^i \\
&\quad - \frac{\bt^2}{\af^4}(\kp_1+\kp_2)F^\mu_{\gm\,ij}y^\gm y^\nu\,\om^{n+\nu}\w\om^j ~, \\
(-\om_i^j\w\om_j^{n+\mu})|_p &= \frac{\bt^2}{\af^4}\kp_1 F^\nu_{\gm\,ij} y^\gm y^\mu\,\om^{n+\nu}\w\om^j + \frac{\bt^2}{4\af^4} F^\nu_{\gm\,ik}y^\gm F^\mu_{\eta\,jk}y^\eta\,\om^{n+\nu}\w\om^j~, \\
(-\om_i^{n+\nu}\w\om_{n+\nu}^{n+\mu})|_p &= \frac{4\bt^2}{\af^4}\kp_1\kp_2s\,\om^{n+\mu}\w\om^i - \frac{4\bt^2}{\af^4}\kp_1\kp_2y^\mu y^\nu\,\om^{n+\nu}\w\om^i + \frac{\bt^2}{\af^4}\kp_2 F^\nu_{\gm\,ij}y^\gm y^\mu\,\om^{n+\nu}\w\om^j ~,
\end{align*}
and then
\begin{align} \begin{split}
R_{(n+\mu)i(n+\nu)j} &= -\left(\frac{2}{\af^2}\kp_1 - \frac{4\bt^2}{\af^4}\kp_1\kp_2s\right)\dt_{\mu\nu}\dt_{ij} + \frac{1}{2\af^2}F^\mu_{\nu\,ij} + \frac{4\bt^2}{\af^4}\kp_1^2y^\mu y^\nu\dt_{ij} \\
&\quad +\frac{\bt^2}{\af^4}(\kp_1+\kp_2)(y^\mu F^\nu_{\gm\,ij} - y^\nu F^\mu_{\gm\,ij})y^\gm + \frac{\bt^2}{4\af^4}y^\gm F^\nu_{\gm\,ik}F^\mu_{\eta\,jk}y^\eta ~.
\end{split} \label{bd_curv6} \end{align}

For the curvature component $\FR_{n+\nu}^{n+\mu}$,
\begin{align*}
(\dd\om_{n+\nu}^{n+\mu})|_p &= \frac{1}{2\af^2}F^\mu_{\nu\,ij}\,\om^i\w\om^j + \frac{\bt^2}{\af^4}\kp_2(y^\mu F^\nu_{\gm\,ij} - y^\nu F^\mu_{\gm\,ij})y^\gm\,\om^i\w\om^j \\
&\quad + \frac{4}{\af^2}\kp_2\,\om^{n+\mu}\w\om^{n+\nu} - \frac{8\bt^2}{\af^4}\kp_2(\kp_1+\kp_2)y^\gm\,\om^{n+\gm}\w(y^\mu\,\om^{n+\nu} - y^\nu\,\om^{n+\mu}) ~, \\
(-\om_{n+\nu}^i\w\om_i^{n+\mu})|_p &= \frac{\bt^2}{4\af^4}y^\gm F^\nu_{\gm\,ik}F^\mu_{\eta\,jk}y^\eta\,\om^i\w\om^j - \kp_1\frac{\bt^2}{\af^4}(y^\nu F^\mu_{\gm\,ij} - y^\mu F^\nu_{\gm\,ij})y^\gm\,\om^i\w\om^j ~, \\
(-\om_{n+\nu}^{n+\gm}\w\om_{n+\gm}^{n+\mu})|_p &= \frac{4\bt^2}{\af^4}\kp_2^2 \left( s\,\om^{n+\nu}\w\om^{n+\mu} - y^\nu y^\gm\,\om^{n+\gm}\w\om^{n+\mu} - y^\mu y^\gm\,\om^{n+\nu}\w\om^{n+\gm} \right) ~,
\end{align*}
and
\begin{align} \begin{split}
R_{(n+\mu)(n+\nu)(n+\gm)(n+\eta)} &= \left(\frac{4}{\af^2}\kp_2 - \frac{4\bt^2}{\af^4}\kp_2^2s \right)(\dt_{\mu\gm}\dt_{\nu\eta} - \dt_{\mu\eta}\dt_{\nu\gm}) \\
&\quad + \frac{4\bt^2}{\af^4}\kp_2(2\kp_1 + \kp_2)(y^\nu y^\gm\dt_{\mu\eta} - y^\nu y^\eta\dt_{\mu\gm} + y^\mu y^\eta\dt_{\nu\gm} - y^\mu y^\gm\dt_{\eta\nu}) ~.
\end{split} \label{bd_curv7} \end{align}

\end{appendix}


\end{document}